\documentclass[oneside,english,american]{amsart}
\usepackage[T1]{fontenc}
\usepackage[latin9]{inputenc}
\usepackage{babel}
\usepackage{textcomp}
\usepackage{mathrsfs}
\usepackage{amstext}
\usepackage{amsthm}
\usepackage{amssymb}
\usepackage{esint}
\usepackage[unicode=true,
 bookmarks=true,bookmarksnumbered=true,bookmarksopen=false,
 breaklinks=false,pdfborder={0 0 1},backref=false,colorlinks=false]
 {hyperref}
\hypersetup{pdftitle={From Kähler-Ricci flow to moving free boundaries},
 pdfauthor={Robert J. Berman, Chinh H. Lu}}
\usepackage{breakurl}

\makeatletter
\numberwithin{equation}{section}
\numberwithin{figure}{section}
\theoremstyle{plain}
\newtheorem{thm}{\protect\theoremname}[section]
  \theoremstyle{remark}
  \newtheorem{rem}[thm]{\protect\remarkname}
  \theoremstyle{plain}
  \newtheorem{prop}[thm]{\protect\propositionname}
  \theoremstyle{plain}
  \newtheorem{lem}[thm]{\protect\lemmaname}
  \theoremstyle{definition}
  \newtheorem{example}[thm]{\protect\examplename}
  \theoremstyle{plain}
  \newtheorem{cor}[thm]{\protect\corollaryname}

\def\makebbb#1{
    \expandafter\gdef\csname#1\endcsname{
        \ensuremath{\Bbb{#1}}}
}\makebbb{R}\makebbb{N}\makebbb{Z}\makebbb{C}\makebbb{H}\makebbb{E}\makebbb{H}\makebbb{P}\makebbb{B}\makebbb{Q}\makebbb{E}

\usepackage{babel}

\usepackage{babel}

\makeatother

\usepackage{babel}

\newcommand{\MA}{{\rm MA}}
\newcommand{\AM}{{\rm AM}}

\makeatother

  \addto\captionsamerican{\renewcommand{\corollaryname}{Corollary}}
  \addto\captionsamerican{\renewcommand{\examplename}{Example}}
  \addto\captionsamerican{\renewcommand{\lemmaname}{Lemma}}
  \addto\captionsamerican{\renewcommand{\propositionname}{Proposition}}
  \addto\captionsamerican{\renewcommand{\remarkname}{Remark}}
  \addto\captionsamerican{\renewcommand{\theoremname}{Theorem}}
  \addto\captionsenglish{\renewcommand{\corollaryname}{Corollary}}
  \addto\captionsenglish{\renewcommand{\examplename}{Example}}
  \addto\captionsenglish{\renewcommand{\lemmaname}{Lemma}}
  \addto\captionsenglish{\renewcommand{\propositionname}{Proposition}}
  \addto\captionsenglish{\renewcommand{\remarkname}{Remark}}
  \addto\captionsenglish{\renewcommand{\theoremname}{Theorem}}
  \providecommand{\corollaryname}{Corollary}
  \providecommand{\examplename}{Example}
  \providecommand{\lemmaname}{Lemma}
  \providecommand{\propositionname}{Proposition}
  \providecommand{\remarkname}{Remark}
\providecommand{\theoremname}{Theorem}

\begin{document}

\title{From the Kähler-Ricci flow to moving free boundaries and shocks }

\author{Robert J. Berman \& Chinh H. Lu}

\address{R.J. Berman, Chalmers University of Technology, robertb@chalmers.se. }

\address{C.H. Lu, Scuola Normale Superiore Pisa, chinh.lu@sns.it.}

\date{\today}
\begin{abstract}
We show that the twisted Kähler-Ricci flow on a complex manifold $X$
converges to a flow of moving free boundaries, in a certain scaling
limit. This leads to a new phenomenon of singularity formation and
topology change which can be seen as a complex generalization of the
extensively studied formation of shocks in Hamilton-Jacobi equations
and hyperbolic conservation laws (notably, in the adhesion model in
cosmology). In particular we show how to recover the Hele-Shaw flow
(Laplacian growth) of growing 2D domains from the Ricci flow. As we
briefly indicate the scaling limit in question arises as the zero-temperature
limit of a certain many particle system on $X.$ 
\end{abstract}

\thanks{We are grateful to David Witt-Nyström for illuminating discussions
on the Hele-Shaw flow. The first named author was supported by grants
from the Swedish Research Council, the European Research Council (ERC)
and the Knut and Alice Wallenberg foundation. The second named author
was supported by the SNS grant \textquotedblleft Pluripotential Theory
in Differential Geometry\textquotedblright{} and the ERC-grant.}

\maketitle
\tableofcontents{}

\section{Introduction}

The celebrated Ricci flow 

\begin{equation}
\frac{\partial g(t)}{\partial t}=-2\mbox{\ensuremath{\mbox{Ric }g(t)}},\label{eq:ricci flow}
\end{equation}
can be viewed as a diffusion type evolution equation for Riemannian
metrics $g(t)$ on a given manifold $X.$ In fact, as described in
the introduction of \cite{ha}, this was one of the original motivations
of Hamilton for introducing the flow. The point is that, locally,
the principal term of minus the Ricci curvature $g$ of a Riemannian
metric is the Laplacian of the tensor $g$ (which ensures the short-time
existence of the flow). The factor $2$ is just a matter of normalization
as it can be altered by rescaling the time parameter by a positive
number $\beta.$ However, this symmetry is broken when a source term
$\theta$ is introduced in the equation:
\begin{equation}
\frac{\partial g(t)}{\partial t}=-\frac{1}{\beta}\mbox{\ensuremath{\mbox{Ric }g(t)+\theta}},\label{eq:twisted ricci flow intro}
\end{equation}
 where $\theta$ is an appropriate symmetric two tensor on $X$ and
$\beta$ thus plays the role of the inverse diffusion constant or
equivalently, the\emph{ }inverse\emph{ temperature} (according to
the ``microscopic'' Brownian motion interpretation of diffusions).
In general terms the main goal of the present paper is to study the
corresponding zero-temperature limit $\beta\rightarrow\infty$ of
the previous equation. An important feature of the ordinary Ricci
flow \ref{eq:ricci flow} is that it will typically become singular
in a finite time, but in some situations (for example when $X$ is
a three manifold, as in Perelman's solution of the Poincaré conjecture)
the flow can be continued on a new manifold obtained by performing
a suitable topological surgery of $X.$ In our setting it turns out
that a somewhat analogous phenomenon of topology change appears at
a finite time $T_{*}$ in the limit $\beta\rightarrow\infty,$ even
if one assumes the long time existence of the flows for any finite
$\beta$. 

More precisely, following \cite{ca,ts,t-z} we will consider the complex
geometric framework where $X$ is a complex manifold, i.e. it is endowed
with a complex structure $J$ and the initial metric $g_{0}$ is Kähler
with respect to $J.$ We will identify symmetric two-tensors and two-forms
of type $(1,1)$ on $X$ using $J$ in the usual way - then the Kähler
condition just means that the form defined by $g_{0}$ is closed.
We will also assume that $\theta$ defines a closed (but not necessarily
semi-positive) form. Then it is well-known that the corresponding
flow $g^{(\beta)}(t)$ emanating from the fixed metric $g_{0}$ preserves
the Kähler property as long as it exists - it is usually called the
\emph{twisted Kähler-Ricci flow} in the literature (and $\theta$
is called the\emph{ twisting form}); see \cite{ca,ts,t-z,c-s,gz2}
and references therein. For simplicity we will also assume that 
\[
-\frac{1}{\beta}c_{1}(X)+[\theta]\geq0
\]
as $(1,1)-$cohomology classes (where $c_{1}(X)$ denotes the first
Chern class of $X)$ which ensures that the flows $g^{(\beta)}(t)$
exist for all positive times \cite{t-z}. Our main result says that
$g^{(\beta)}(t)$ admits a unique (singular) limit $g(t)$ as $\beta\rightarrow\infty,$
where $g(t)$ defines a positive current with $L^{\infty}-$coefficients: 
\begin{thm}
\label{thm:main intro}Let $X$ be a compact complex manifold endowed
with a smooth form $\theta.$ Then 
\[
\lim_{\beta\rightarrow\infty}g^{(\beta)}(t)=P(g_{0}+t\theta)\:\:(:=g(t))
\]
in the weak topology of currents, where $P$ is a (non-linear) projection
operator onto the space of positive currents. Moreover, the metrics
$g^{(\beta)}(t)$ are uniformly bounded on any fixed time interval
$[0,T].$
\end{thm}
The definition of the projection operator $P$ will be recalled in
Section \ref{sub:The-projection-operator}. The point is that the
linear curve $g_{0}+t\theta,$ which coincides with the limiting flow
for short times will, unless $\theta\geq0,$ leave the space of Kähler
forms at the time 
\begin{equation}
T_{*}:=\sup\{t:\,g_{0}+t\theta\geq0\}\label{eq:def of T star intro}
\end{equation}
and hence it cannot be the limit of the metrics $g^{(\beta)}(t),$
even in a weak sense, for $t>T_{*}.$ In particular, this means that
around the time $T_{*}$ the Ricci curvatures $\mbox{Ric }g^{(\beta)}(t)$
will become unbounded as $\beta\rightarrow\infty$ (indeed, otherwise
one could neglect the first term in the equation \ref{eq:twisted ricci flow intro}
to obtain a linear ODE in the large $\beta-$limit solved by $g_{0}+t\theta)$.
Still we will show that the metrics $g^{(\beta)}(t)$ do remain uniformly
bounded from above as $\beta\rightarrow\infty.$ However, unless $\theta>0,$
the limiting $L^{\infty}-$metrics $g(t)$ will, for $t\geq T_{*},$
degenerate on large portions of $X,$ i.e. the support 
\[
X(t):=\mbox{supp\ensuremath{(dV_{g(t)})}}
\]
of the limiting $L^{\infty}-$volume form $dV_{g(t)}$ is a proper
closed subset of $X$ evolving with $t.$ Moreover, on the support
$X(t)$ the metrics $g(t)$ do evolve linearly, or more precisely
\[
g(t)=g_{0}+t\theta\,\,\,\mbox{on\,\ensuremath{X(t),}}
\]
in the almost everywhere sense. As a consequence, typically the volume
form $dV_{g(t)}$ has a sharp discontinuity over the boundary of $X(t),$
showing that the limiting (degenerate) $L^{\infty}-$metrics $g(t)$
are not continuous and hence $C^{0}-$convergence in the previous
theorem cannot hold, in general. In the generic case the evolving
open sets $\Omega(t):=X-X(t),$ where $dV_{g(t)}$ vanishes identically
are increasing and may be characterized as solutions of moving free
boundary value problems for the complex Monge-Ampère equation (see
Section \ref{sub:The-projection-operator}). 

The projection operator $P$ appearing in the theorem, which associates
to a given $(1,1)$ current $\eta$ on $X$ a positive current, cohomologous
to $\eta,$ is defined as a (quasi) plurisubharmonic envelope on the
level of potentials (Section \ref{sub:The-projection-operator}) and
can be viewed as a complex generalization of the convex envelope of
a function. Such envelopes play a key role in pluripotential theory
(as further discussed in Section \ref{sub:Further-relations-to},
below). In particular, the previous theorem yields a dynamic PDE construction
of the envelopes in question, giving an alternative to previous dynamic
constructions appearing in the real convex analytical setting \cite{ve,g-c}
(see the discussion in Section \ref{sub:The-case-when convex envelopes}).

More generally, the weak convergence in Theorem \ref{thm:main intro}
will be shown to hold as long as the $\theta$ (viewed as a current)
has continuous potentials. But then the limit $g(t)$ will, in general,
not be in $L^{\infty}$ (unless $\theta$ is). Moreover, the support
of the corresponding measure $dV_{g(t)}$ may then be a subset of
low Hausdorff dimension. For example, in the one dimensional setting
appearing in the adhesion model discussed below the conjectures formulated
in \cite{saf} suggest an explicit formula for the Haussdorf dimension
of the support, at any given time $t,$ when $\theta$ is taken as
a random Gaussian distribution with given scaling exponent (see Section
\ref{rem:random potential}). 

We will pay a particular attention to the special case in Theorem
\ref{thm:main intro} where the twisting form $\theta$ represents
the trivial cohomology class, i.e. 
\[
\theta=dd^{c}f,
\]
 for a function $f$ on $X.$ The large $\beta-$limit of the corresponding
twisted Kähler-Ricci flow turns out to be intimately related to various
growth processes appearing in mathematical physics (and hence the
Kähler-Ricci flow can be used as a new regularization of such processes):

\subsubsection*{Hamilton-Jacobi equations, shock propagation and the adhesion model
in cosmology}

In the particular case when $X$ is an abelian variety (or more specifically
$X=\C^{n}/\Lambda+i\Z^{n},$ for a lattice $\Lambda$ in $\R^{n})$
and the potential $f$ of the twisting form $\theta$ is invariant
along the imaginary direction, we will show that the corresponding
limiting twisted Kähler-Ricci flow $g(t)$ corresponds, under Legendre
transformation in the space variables, to a viscosity solution $u(x,t)$
of the\emph{ }Hamilton-Jacobi equation\emph{ }in $\R^{n}$ with periodic
Hamiltonian $f$ \cite{c-l,b-e,l-r}. Under this correspondence the
critical time $T_{*}$ (formula \ref{eq:def of T star intro}) corresponds
to the first moment of shock (caustic) formation in the solution $u_{t}(x),$
i.e. the time where $u_{t}$ ceases to be differentiable. From this
point of view the moving domains $\Omega(t)$ correspond, under Legendre
duality, to the evolving shock hypersurfaces $S_{t}$ (i.e the non-differentiability
locus of $u_{t})$. The evolution and topology change of such shocks
plays a prominent role in various areas of mathematical physics (and
more generally fit into the general problem of singularity formation
in hyperbolic conservation laws \cite{se}). In particular, the evolving
shock hypersurface $S_{t}$ model the concentration of mass density
in the cosmological adhesion model describing the formation of large-scale
structures during the early expansion of the universe \cite{v-d-f-n,g-m-s,hsw,hsw2}.
Our setting contains, in particular, the case when the initial data
in the adhesion model is periodic \cite{k-p-s-m,hsw,hsw2}. It should
also be pointed out that in this picture the limit $\beta\rightarrow\infty$
can be seen as a non-linear version of the classical vanishing viscosity
limit \cite{c-l,b-e,l-r}, which has the virtue of preserving convexity.

We will also study the corresponding large time limits and show that
if the set $F$ of absolute minima of the potential $f$ is finite,
then the support of the positive current defined by the joint large
$\beta$ and large $t-$ limit of the twisted Kähler-Ricci flow is
a piecewise affine hypersurface whose vertices coincides with $F$
and whose lift to $\R^{n}$ gives a Delaunay type tessellation of
$\R^{n}$ (which is consistent with numerical simulations appearing
in cosmology \cite{k-p-s-m,hsw,hsw2}).

\subsubsection*{Applications to the Hele-Shaw flow (Laplacian growth)}

In another direction, allowing $\theta$ to be a singular current
of the form 
\[
\theta=\omega_{0}-[E],
\]
where $\omega_{0}$ is the initial Kähler form and $[E]$ denotes
the current of integration along a given effective divisor (i.e. complex
hypersurface) in $X$ cohomologous to $\omega_{0},$ we will show
that the corresponding domains $\Omega(t),$ which in this setting
are growing continuously with $t,$ give rise to a higher dimensional
generalization of the classical Hele-Shaw flow in a two-dimensional
geometry. More precisely, the Hele-Shaw flow appears when $X$ is
a Riemann surface, $\omega_{0}$ is normalized to have unit area and
$E$ is given by a point $p$ (in the classical setting $X$ is the
Riemann sphere and $p$ is the point at infinity; the general Riemann
surface case was introduded in \cite{hs}). Then $\Omega(t)$ coincides,
up to a time reparametrization, with the\emph{ }Hele-Shaw flow (also
called Laplacian growth) injected at the point $p$ in the medium
$X$ with varying permeability (encoded in the form $\omega_{0}).$
The latter flow was originally introduced in fluid mechanics to model
the expansion of an incompressible fluid $\Omega(t)$ of high viscosity
(for example oil) injected at a constant rate in another fluid of
low viscosity (such as water) occupying the decreasing region $X(t).$
In more recent times the Hele-Shaw flow has made its appearance in
various areas such as random matrix theory, integrable system and
the Quantum Hall Effect \cite{z} to name a few (see \cite{va} for
a historical overview). In particular, in the latter setting $X(t)$
represents the electron droplet. Special attention has been payed
to an interesting phenomenon of topology change in the flow appearing
at the time where $\Omega(t)$ becomes singular (which is different
from $T_{*}$ which in this singular setting vanishes). Various approaches
have been proposed to regularize the Hele-Shaw flow in order to handle
the singularity formation (see \cite[Section 5.3]{va}). The present
realization of the Hele-Shaw flow from the limit of the Kähler metrics
$\omega^{(\beta)}(t)$ on $X-\{p\}$ suggest a new type of regularization
scheme, for example using the corresponding thick-thin decomposition
of $X,$ as in the ordinary Ricci flow (with $X(t)$ and $\Omega(t)$
playing the role of the limiting thick and thin regions, respectively).
But we will not go further into this here.

\subsection{\label{sub:Further-relations-to}Further relations to previous results }

There is an extensive and rapidly evolving literature on the Kähler-Ricci
flow (and its twisted versions) starting with \cite{ca}; see for
example \cite{s-w} and references therein. But as far as we know
the limit $\beta\rightarrow\infty$ (which is equivalent to scaling
up the twisting form and rescaling time) has not been studied before.
For a finite $\beta$ there is no major analytical difference between
the Kähler-Ricci flow and its twisted version, but in our setting
one needs to make sure that the relevant geometric quantities do not
blow up with $\beta$ (for example, as discussed above the Ricci curvature
does blow up). For a finite $\beta$ the surgeries in the Kähler-Ricci
flow have been related to the Minimal Model Program in algebraic geometry
in \cite{c-la,s-t}, where the final complex-geometric surgery produces
a minimal model of the original algebraic variety. In Section \ref{sub:Comparison-with-convergence}
we compare some of our results with the corresponding long time convergence
results on the minimal model (which produces canonical metrics of
Kähler-Einstein type) \cite{ts,t-z,s-t-1}. In the algebro-geometric
setting negative twisting currents $\theta$ also appear naturally,
when $X$ is the resolution of a projective variety with canonical
singularities \cite{s-t,egz} $(\theta$ is then current of integration
along minus the exceptional divisor). Recently, viscosity techniques
were introduced in \cite{egz} to produce viscosity solutions for
the twisted Kähler-Ricci flow (and in particular its singular variants
appearing when $\theta$ is singular). But, again, this concerns the
case when $\beta$ is finite.

In the case when $(X,\omega)$ is invariant under the action of a
suitable torus $T$ (i.e. $X$ is a toric variety or an Abelian variety)
the corresponding time dependent convex envelopes (studied in Section\ref{sub:The-case-when convex envelopes})
have recently appeared in \cite{rz,rzIII} in a different complex
geometric than the Kähler-Ricci flow, namely in the study of the Cauchy
problem for weak geodesic rays in the space of Kähler metrics (see
Remark \ref{rem:r-z}). Moreover, in \cite{rwn5,rwn4,rwn3} the Hele-Shaw
flow and the corresponding phenomenon of topology change was exploited
to study the singularities of such weak geodesic rays (and solutions
to closely related homogeneous complex Monge-Ampère equations) in
the general non-torus invariant setting (see Remark \ref{rem:hsf gives weak geod}).

We also recall that envelope type constructions as the one appearing
in the definition of the projection operator $P$ play a pivotal role
in pluripotential theory (and have their origins in the classical
work of Siciak and Zakharyuta on polynomial approximations in $\C^{n}$
(see \cite{gz} for the global setting). Moreover, by the results
in \cite{bbgz} the corresponding measure $(P\theta)^{n}$ on $X$
can be characterized as the unique normalized minimizer of the (twisted)
pluripotential energy (which generalizes the classical weighted logarithmic
energy of a measure in $\C).$ The $L^{\infty}-$regularity of $P\theta$
was first established in \cite{b-d} in a very general setting (of
big cohomology classes), using pluripotential techniques. A new PDE
proof of the latter regularity, in the case of nef and big cohomology
classes, was then given in the paper \cite{berm5}, which can be seen
as the ``static'' version of the present paper.

\subsection{Organization}

In Section \ref{sec:The-zero-temperature-limit} we state and prove
refined versions of Theorem \ref{thm:main intro} (stated above).
Then in Section \ref{sec:Large-time-asymptotics} we go on to study
the joint large $\beta$ and large $t-$limits of the corresponding
flows. In particular, a dynamical construction of plurisubharmonic
(as well as convex) envelopes is given and a comparison with previous
work on canonical metrics in Kähler geometry (concerning finite $\beta)$
is made. In Sections \ref{sec:Relations-to-viscosity} and \ref{sec:Application-to-Hele-Shaw}
the relation to Hamilton-Jacobi equations and Hele-Shaw flows, respectively,
is exhibited. The extension to twisting potentials which are merely
continuous and the relation to random twistings is discussed in Section
\ref{rem:random potential}. In the final section we present a (deterministic,
as well as stochastic) gradient flow interpretation of our results
which will be expanded on elsewhere.

\section{\label{sec:The-zero-temperature-limit}The zero-temperature limit
of the Kähler-Ricci flow }

\subsection{Notation and setup}

Let $X$ be an $n$-dimensional compact complex manifold. We will
identify symmetric two-tensors with two-forms of type $(1,1)$ on
$X$ using $J$ in the usual way: if $g$ is a symmetric tensor, then
the corresponding form $\omega:=g(\cdot,J\cdot),$ is said to be \emph{Kähler}
if $\omega$ is closed and $g$ is strictly positive (i.e. $g$ is
a Riemannian metric). We will assume that $X$ is Kähler, i.e. it
admits a Kähler metric and we fix such a reference Kähler metric $\omega$
once and for all. On a Kähler manifold the De Rham cohomology class
$[\eta]\in H^{2}(X,\R)$ defined by a given closed real two form $\eta$
of type $(1,1)$ may (by the ``$\partial\bar{\partial}-$lemma'')
be written as 
\begin{equation}
[\eta]=\left\{ \eta+dd^{c}u:\,\,\,\,u\in C^{\infty}(X)\right\} ,\,\,\,dd^{c}:=\frac{i}{2\pi}\partial\bar{\partial.}\label{eq:eta as ddbar}
\end{equation}
In our normalization the Ricci curvature form $\mbox{Ric }\omega$
of a Kähler metric $\omega$ on $X$ is defined, locally, by 
\[
\mbox{Ric }\omega:=-dd^{c}\log\frac{\omega^{n}}{dV(z)}
\]
 where $z$ are local holomorphic coordinates on $X$ and $dV(z)$
denotes the corresponding Euclidean volume. The form $\mbox{Ric }\omega$
represents, for any Kähler metric $\omega,$ minus the first Chern
class $c_{1}(K_{X})\in H^{2}(X,\R)$ of the canonical line bundle
$\det(T^{*}X).$

\subsubsection{Setup}

Specifically, our geometric setup is as follows: we assume given a
family $\theta_{\beta}$ of closed real $(1,1)-$forms (the ``twisting
forms'') with the asymptotics 
\[
\theta_{\beta}=\theta+o(1),
\]
 as $\beta\rightarrow\infty$ (in $L^{\infty}$-norm). We will assume
that 

\begin{equation}
c_{1}(K_{X})/\beta+[\theta_{\beta}]\geq0\label{eq:semi-positivity cond in text}
\end{equation}
 as $(1,1)-$cohomology classes, i.e. there exists a semi-positive
form $\chi_{\beta}$ in the class $c_{1}(K_{X})/\beta+[\theta_{\beta}]$
(we will fix one such choice for each $\beta>0$). This assumption
ensures that the corresponding twisted Kähler-Ricci flow 
\begin{equation}
\frac{\partial\omega(t)}{\partial t}=-\frac{1}{\beta}\mbox{\ensuremath{\mbox{Ric }\omega(t)+\theta_{\beta}}},\,\,\,\,\omega(0)=\omega_{0}\label{eq:nonnormalized krf in setup}
\end{equation}
exist for all $t\geq0$ and $\beta>0$ \footnote{In fact, it is enough to assume that $c_{1}(K_{X})/\beta+[\theta_{\beta}]$
is nef (i.e. a limit of positive classes), which is equivalent to
the long time existence of the corresponding KRF \cite{t-z}. Indeed,
the estimates we get will be independent of the choice of reference
form $\chi$ and hence the nef case can be reduced to the semi-positive
case by perturbation of the class $[\theta]$. }. The extra flexibility offered by $\beta-$dependence of $\theta_{\beta}$
will turn out to be quite useful (for example, taking $\theta_{\beta}:=\theta+\frac{1}{\beta}\mbox{Ric \ensuremath{\omega}}$
for $\theta$ defining a semi-positive cohomology class, ensures that
the semi-positivity condition \ref{eq:semi-positivity cond in text}
holds). 

More precisely, we will refer to the flow above as the \emph{non-normalized
twisted Kähler-Ricci flow} (or simply the \emph{non-normalized KRF})
to distinguish it from its normalized version: 
\begin{equation}
\frac{\partial\omega(t)}{\partial t}=-\frac{1}{\beta}\mbox{\ensuremath{\mbox{Ric }\omega(t)-\omega(t)+\theta_{\beta},\,\,\,\,\omega(0)=\omega_{0}}}\label{eq:normalized krf in setup}
\end{equation}

As is well-known the two flows are equivalent under a scaling combined
with a time reparametrization: denoting by $\tilde{\omega}(s)$ the
non-normalized KRF one has
\begin{equation}
\frac{1}{s+1}\tilde{\omega}(s)=\omega(t),\,\,\,\,e^{t}:=s+1,\label{eq:scaling of k=0000E4hler forms}
\end{equation}
(the equivalence follows immediately from the fact that $\mbox{Ric }(c\omega)=\mbox{Ric }\omega,$
for any given positive constant $c).$ In the proofs we will reserve
the notation $\omega(t)$ for the normalized version of the flows. 

In order to write the flows in terms of Kähler potentials we represent
\[
\tilde{\omega}(s)=(\omega_{0}+s\chi_{\beta})+dd^{c}\tilde{\varphi}(s),\,\,\,\,\tilde{\varphi}(0)=0
\]
for the fixed semi-positive form $\chi_{\beta}$ in $\frac{1}{\beta}c_{1}(K_{X})+[\theta_{\beta}]$
(the first term ensures that the equation holds on the level of cohomology).
Then the non-normalized KRF \ref{eq:nonnormalized krf in setup} is
equivalent to the following Monge-Ampère flow: 
\[
\frac{\partial\tilde{\varphi}(s)}{\partial s}=\frac{1}{\beta}\log\frac{(\tilde{\omega}_{0}+s\chi_{\beta}+dd^{c}\tilde{\varphi}(s))^{n}}{\omega^{n}}+f_{\beta},\,\,\tilde{\varphi}(0)=0
\]
for the smooth function $\tilde{\varphi}(s),$ which is a Kähler potential
of $\tilde{\omega}(s)$ wrt the Kähler reference metric $\omega_{0}+s\chi_{\beta}$,
and where $f_{\beta}$ is uniquely determined by the equation 

\begin{equation}
\theta_{\beta}-\frac{1}{\beta}\mbox{Ric }\omega=dd^{c}f_{\beta}+\chi_{\beta}\label{eq:defining eq for f beta in terms of theta}
\end{equation}
together with the normalization condition 

\begin{equation}
\inf_{X}f_{\beta}=0.\label{eq:normalization of f beta}
\end{equation}
We will also assume that $\chi_{\beta}$ is uniformly bounded from
above, i.e. 
\begin{equation}
\chi_{\beta}\leq C_{0}\omega,\label{eq:condition on chi beta}
\end{equation}
 for some constant $C_{0}$ and some fixed Kähler form $\omega$,
and hence $\chi_{\beta}$ converge smoothly to $\chi\in[\theta]$
as $\beta\to+\infty$ \footnote{The convergence result still holds without an upper bound on $\chi_{\beta}$
(which does not hold when $\theta$ is nef but not semi-positive),
but the dependence on $t$ in the estimates will be worse. }. Accordingly, $f_{\beta}$ converge smoothly to $f$ uniquely determined
by $\chi+dd^{c}f=\theta$ and $\inf_{X}f=0$. Since $\omega_{0}$
is smooth, up to enlarging $C_{0}$ we can also assume that 
\begin{equation}
C_{0}^{-1}\omega\text{\ensuremath{\leq}}\omega_{0}\leq C_{0}\omega.\label{eq:omega zero and omega}
\end{equation}
Finally, even when the functions $f_{\beta}$ are not uniformly bounded,
Lemma \ref{lem:proj of plus infity} ensures that the envelope $P_{\hat{\omega}_{t}}(f_{\beta})$
stays bounded from above if the functions $f_{\beta}$ do not go uniformly
to $+\infty$. After enlarging $C_{0}$ one more time we can assume
that 
\begin{equation}
P_{\hat{\omega_{t}}}(f_{\beta})\leq C_{0},\:\forall\beta>0,\:\forall t\geq0.\label{eq:upper bound for envelope of f beta}
\end{equation}

Similarly, the normalized KRF is equivalent to the Monge-Ampère flow
\[
\frac{\partial\varphi(t)}{\partial t}=\frac{1}{\beta}\log\frac{(\hat{\omega}_{t}+dd^{c}\varphi(t))^{n}}{\omega^{n}}-\varphi(t)+f_{\beta},
\]
 where 
\[
\omega(t)=\hat{\omega}_{t}+dd^{c}\varphi(t),\,\,\,\hat{\omega}_{t}:=e^{-t}\omega_{0}+(1-e^{-t})\chi_{\beta}.
\]
 The corresponding scalings are now given by

\[
\tilde{\varphi}(s)=e^{t}\left(\varphi(t)+c_{\beta}(t)\right),\,\,\,c_{\beta}(t)=\frac{n}{\beta}(t-1+e^{-t})
\]
 (abusing notation slightly we will occasionally also write $\omega(t)=\omega_{\varphi_{t}}).$ 
\begin{rem}
\label{rem:scaling of linear evolution}Under the scaling above a
curve $\tilde{\varphi}(s)$ of the form $\tilde{\varphi}(s)=\varphi_{0}+sf$
corresponds to a curve $\varphi(t)$ of the form $e^{-t}\varphi_{0}+(1-e^{-t})f-c_{\beta}(t).$ 
\end{rem}

\subsection{Statement of the main results}

In the following section we will prove the following more precise
version of Theorem \ref{thm:main intro} stated in the introduction
of the paper.
\begin{thm}
\label{thm:main general} Let $X$ be a compact complex manifold endowed
with a family of twisting form $\theta_{\beta}$ as above. Denote
by $\omega^{(\beta)}(t)$ the flow of Kähler metrics evolving by the
(non-normalized) twisted Kähler-Ricci flow \ref{eq:nonnormalized krf in setup}
with parameter $\beta,$ emanating from a given Kähler metric $\omega_{0}$
on $X.$ Then 
\[
\lim_{\beta\rightarrow\infty}\omega^{(\beta)}(t)=P(\omega_{0}+t\theta)
\]
in the weak topology of currents. On the level of Kähler potentials,
for any fixed time-interval $[0,T]$, the functions $\varphi^{(\beta)}(t)$
converge uniformly wrt $\beta$ in the $C^{1,\alpha}(X)$-topology
(for any fixed $\alpha<1)$ towards the envelope $P_{\omega_{0}+t\theta}(0)$.
More precisely, fixing a reference Kähler metric $\omega$ on $X$,
\[
0\leq\omega^{(\beta)}(t)\leq e^{C(1+\frac{1}{\beta})t\log(t+1)}\omega
\]
and 
\[
-C-n\log(1+t))/\beta\leq\frac{\partial\varphi^{(\beta)}(t)}{\partial t}\leq C(1+\frac{1}{\beta}\log(1+t))/t
\]
where the constant $C$ only depends on $\theta$ through the following
quantities: $\sup_{X}\mbox{Tr}_{\omega}\theta$ and $C_{0}$ (as in
\ref{eq:condition on chi beta}, \ref{eq:omega zero and omega} and
\ref{eq:upper bound for envelope of f beta}); it also depends on
a lower bound on the holomorphic bisectional curvature of the reference
Kähler metric $\omega.$ 
\end{thm}
The definition of the non-linear projection operators $P$ and $P_{\omega_{0}}$
will be recalled in Section \ref{sub:The-projection-operator}. The
dependence of the constants above on the potential $f$ of $\theta$
will be crucial in the singular setting of Hele-Shaw type flows where
$f$ blows up on a hypersurface of $X,$ but $P_{C'\omega}(f)$ is
finite (see Section \ref{sec:Application-to-Hele-Shaw}). 

Under special assumptions on $X$ we get an essentially optimal bound
on $\omega^{(\beta)}(t):$ 
\begin{thm}
\label{thm:sharp bounds}Assume that $X$ admits a Kähler metric $\omega$
with non-negative holomorphic bisectional curvature. Then the following
more precise estimates hold
\[
\left\Vert \omega^{(\beta)}(t)\right\Vert \leq(t+1)\max\left\{ \left\Vert \omega_{0}\right\Vert ,\left\Vert \theta_{\beta}-\frac{\mbox{Ric }\omega}{\beta}\right\Vert \right\} 
\]
 in terms of the trace norm defined wrt $\omega$ (i.e the sup on
$X$ of the point-wise $L^{1}-$norm wrt $\omega$. Moreover, for
any Riemann surface (i.e. $n=1)$ the previous estimate holds without
any conditions on the Kähler metric $\omega.$ 
\end{thm}
In particular, letting $\beta\rightarrow\infty$ gives that 
\[
\left\Vert P(\omega_{0}+t\theta)\right\Vert \leq(t+1)\max\{\left\Vert \omega_{0}\right\Vert ,\left\Vert \theta\right\Vert \},
\]
 which is also a consequence of the estimates in the ``static''
situation considered in \cite{berm5} (of course, in the case when
$\theta$ is semi-positive the latter bound follows directly from
the triangle inequality!). 

However, it should be stressed that, in general, it is not possible
to bound $\omega^{(\beta)}(t)$ by a factor $C_{\beta}t,$ even for
a fixed $\beta$ (see Prop \ref{prop:no linear bd}). On the other
hand, as we show in Section \ref{sub:The-case-when convex envelopes}
this is always possible if $[\theta]=[\omega_{0}]$ (and in particular
positive).

\subsection{\label{sub:Preliminaries}Preliminaries}

\subsubsection{\label{sub:Parabolic-comparison/max-princip}Parabolic comparison/max
principles}

We will make repeated use of standard parabolic comparison and maximum
principles for smooth sub/super solutions of parabolic problems of
the form
\[
\frac{\partial u}{\partial t}=\mathcal{D}u
\]
for a given differential operator $\mathcal{D}$ acting on $\mathcal{C}^{\infty}(X)$
(or a subset thereof). We will say that $u$ is a sub (super) solution
if $(\frac{\partial}{\partial t}-\mathcal{D})u\leq0$ ($\geq0).$ 
\begin{prop}
(Comparison principle) Let $X$ be a compact complex manifold and
consider a second order differential operator $\mathcal{D}$ on $\mathcal{C}^{\infty}(X)$
of the form 
\[
(\mathcal{D}u)(x)=a(t,x)u(x)+F_{t}((dd^{c}u)(x)),
\]
 where $a$ is a bounded function on $[0,\infty[\times X$ and $F_{t}(A)$
is a family of increasing functions on the set of all Hermitian matrices.
If $u$ and $v$ are smooth sub- and super-solutions, respectively,
to the corresponding parabolic problem for $\mathcal{D}$ on $X\times[0,T],$
then $u_{0}\leq v_{0}$ implies that $u_{t}\leq v_{t}$ for all $t\in[0,T].$
In particular, the result applies to the heat flow of the time-dependent
Laplacian $\Delta_{g_{t}},$ defined wrt a family of Kähler metrics,
and to the twisted KRF (normalized as well as non-normalized). \end{prop}
\begin{proof}
For completeness (and since we shall need a slight generalization)
we recall the simple proof. After replacing $u$ with $e^{At}u$ for
$A$ sufficiently large we may as well assume that $a_{t}<0.$ Assume
to get a contradiction that it is not the case that $u_{t}\leq v_{t}$
on $X\times[0,T].$ Then there exists a point $(x,t)\in X\times[0,T]$
such that 
\[
u_{t}(x)-v_{t}(x)>0,\,\frac{\partial(u_{t}-v_{t})}{\partial t}(x)\geq0,\,(\nabla_{g_{t}}u)(x)=(\nabla_{g_{t}}v)(x)=0,
\]
and $\,\,(dd^{c}u_{t})(x)-(dd^{c}v_{t})(x)\leq0.$ Indeed, one first
takes $t$ to be the first time violating the condition $u_{t}\leq v_{t}$
on $X$ and then maximize $u_{t}(x)-v_{t}(x)$ over $X$ to get the
point $x.$ In particular, since $a_{t}(x)>0$ and $F_{t}$ is increasing
we have that 
\[
\frac{\partial(u_{t}-v_{t})}{\partial t}(x)-(\mathcal{D}u-\mathcal{D}v)(x)>0.
\]
But this contradicts that $u$ and $v$ are sub/super solutions (since
this implies the reversed inequality $\leq0).$\end{proof}
\begin{rem}
The condition that $X$ be a complex manifold (and the Kähler condition)
have just been included to facilitate the formulation of the proposition.
Moreover, exactly the same proof as above shows that any first order
term of the $H(t,x,(\nabla u)(x))$ for $H$ smooth can be added to
$\mathcal{D}$ above (as in the setting of Hamilton-Jacobi equations
considered in Section \ref{sec:Relations-to-viscosity}).\end{rem}
\begin{prop}
(Maximum principle) Let $X$ be a compact complex manifold and consider
a second order differential operator $\mathcal{D}$ on $\mathcal{C}^{\infty}(X)$
of the form 
\[
(\mathcal{D}u)(x)=F_{t}((dd^{c}u)(x)),
\]
 where $F_{t}(A)$ is a family of increasing functions on the set
of all Hermitian matrices. Given a smooth function $u(x,t)$ on $X\times[0,T]$
we have that 
\begin{itemize}
\item The following dichotomy holds: either the maximum of $u(x,t)$ is
attained at $X\times\left\{ 0\right\} $ or at a point $x\in X\times]0,T]$
satisfying 
\[
\left(\frac{\partial u(x,t)}{\partial t}-\mathcal{D}(u)\right)\geq-F_{t}(0),
\]

\item In particular, if $F_{t}(0)=0$ for all $t$ and 
\[
\left(\frac{\partial}{\partial t}-\mathcal{D}\right)\leq0
\]
 on $X\times[0,T],$ then the maximum of $u(x,t)$ is attained at
$X\times\left\{ 0\right\} $. 
\end{itemize}
\end{prop}
\begin{proof}
The first property is proved exactly as in the beginning of the proof
of the comparison principle. The second point then follows by replacing
$u$ with $u-\delta t$ for any number $\delta>0.$\end{proof}
\begin{rem}
\label{rem:comp for concave}We will need a slight generalization
of the comparison principle to functions $u(x,t)$ which are continuous
on $X\times[0,T]$ and such that $u(\cdot,t)$ is smooth on $X$ for
any fixed $t>0$ and $u(x,\cdot)$ is quasi-concave on $[0,T]$ for
$x$ fixed, i.e. the sum of a concave and a smooth function. Then
we simply define $\frac{\partial}{\partial t}u(x,t)$ on $]0,T]$
as the left derivate i.e. $\frac{\partial}{\partial t}u(x,t):=\lim_{h\rightarrow0}(u(x,t+h)-u(x,t))/h$
for $h<0.$ In particular, the notion of a subsolution still makes
sense for $u$ and the proof of the comparison principle then goes
through word for word. This is just a very special case of the general
notion of viscosity subsolution \cite{c-e-l} which, by definition,
means that the parabolic inequality holds with respect to the super
second order jet of $u$ (which in our setting is just the ordinary
jet in the space$-$direction and the interval between the right and
the left derivative in the time-direction). See \cite{egz} for the
complex setting, where very general comparison principles are established
for viscosity sub/super solution (which however are not needed for
our purposes)
\end{rem}

\subsubsection{\label{sub:The-projection-operator}The projection operator $P$ }

Let $\eta$ be a given closed smooth real $(1,1)-$form on $X$ and
denote by $[\eta]$ the corresponding De Rham cohomology class of
currents which may be represented as in formula \ref{eq:eta as ddbar},
in terms of functions $u\in L^{1}(X).$ Under this representation
the subspace of all positive currents in $[\eta]$ corresponds to
the space of all \emph{$\eta-$plurisubharmonic (psh)} functions $u,$
denoted by $PSH(X,\eta),$ i.e. $u$ is an upper semi-continuous (usc)
function such that 
\[
\eta_{u}:=\eta+dd^{c}u\geq0
\]
 in the sense of currents. We will always assume that $PSH(X,\eta)$
is non-empty (which, by definition, means that the class $[\eta]$
is \emph{pseudo-effective. }This is the weakest notion of positivity
of a class $[\eta]\in H^{1,1}(X,\R),$ the strongest being that $[\eta]$
is a \emph{Kähler class} (also called \emph{positive}), which, by
definition, means that it contains a Kähler metric. 

Given a lsc bounded function $f$ one obtains an $\eta-$psh function
$P_{\eta}(f)$ as the envelope 
\begin{equation}
P_{\eta}(f)(x):=\sup_{u\in PSH(X,\eta)}\{u(x):\,\,\,u\leq f,\,\mbox{on\,\ensuremath{X}}\}.\label{eq:def of psh env}
\end{equation}
 The operator $P_{\eta}$ is clearly a projection operator in the
sense that $P_{\eta}(u)=u$ if $u$ is in $PSH(X,\eta)\cap C^{0}(X).$
We then define 
\[
P(\eta):=\eta+dd^{c}(\,),
\]
 which thus defines a positive current cohomologous to $\eta.$ Equivalently,
if one fixes another reference form $\omega$ in $[\eta],$ i.e. 
\[
\eta=\omega+dd^{c}f
\]
 for some function $f.$ Then 
\[
P(\eta):=\omega+dd^{c}(P_{\omega}f).
\]
If the class $[\eta]$ is semi-positive, i.e. $PSH(X,\eta)\cap C^{\infty}(X)$
is non-empty, then it follows immediately from the definition that
$P_{\eta}(f)$ is bounded if $f$ is. However, even if $f$ is smooth
$P_{\eta}(f)$ will in general not be $C^{2}-$smooth. On the other
hand, by \cite{b-d,berm5} $P_{\eta}(f)$ is almost $C^{2}-$smooth
if the class $[\eta]$ is positive:
\begin{prop}
Let $\omega$ be a Kähler form and $f$ a smooth function on $X.$
Then the complex Hessian $dd^{c}(P_{\omega}f)$ is in $L^{\infty}.$
Equivalently, given any smooth form $\eta$ defining a positive class
$[\eta]$ the corresponding positive current $P(\eta)$ in $[\eta]$
is in $L^{\infty}.$ As a consequence, 
\begin{equation}
P(\eta)^{n}=1_{C}\eta^{n},\label{eq:point wise formula for MA of env in text}
\end{equation}
in the point-wise almost everywhere sense, where $C$ is the corresponding
(closed) coincidence set:
\[
C:=\{x\in X:\,\,\,P_{\eta}(0)(x)=0\}.
\]

\end{prop}
In fact, we will get a new proof of the previous result using the
Kähler-Ricci flow (which can be seen as a dynamic version of the proof
in \cite{berm5}); see Section \ref{sub:The-case-when convex envelopes}. 
\begin{rem}
\label{rem:free}Setting $u:=P_{\omega}f$ and $\Omega:=\{P_{\omega}f<f\}$
the previous proposition implies that the pair $(u,\Omega)$ can be
characterized as the solution to the following free boundary value
problem for the complex Monge-Ampère operator with obstacle $f,$
i.e. $u\leq f$ on $X$ and

\[
(\omega+dd^{c}u)^{n}=0\,\mbox{in\,\ensuremath{\Omega,\,\,\,u=f,\,du=df\,\mbox{on \ensuremath{\partial\Omega}}}}
\]
and $\omega+dd^{c}u\geq0$ on $X.$ In the case when $n=1$ it is
well-known that $u$ is even $C^{1,1}-$smooth \cite{b-k}, but the
free boundary $\partial\Omega$ may be extremely irregular and even
if $\omega$ is real analytic it will, in general, have singularities
\cite{sc}. 
\end{rem}
A key role in the present paper will be played by parametrized envelopes
(where $f$ varies linearly with time). 
\begin{lem}
\label{lem:concave proj}Given functions $\varphi$ and $f$ on $X$
the function $t\mapsto\varphi(t,x):=P_{\omega}(\varphi+tf)(x)$ on
$\R$ is concave for $x$ fixed. Moreover, locally on $]0,\infty[$
the corresponding curve $\varphi(t)$ can be written as a uniform
limit $\varphi_{\epsilon}(t)$ of concave curves with values in $PSH(X,\omega)\cap\mathcal{C}^{\infty}(X).$
Furthermore, if $\frac{\partial\varphi(t)}{\partial t}\leq g$ for
a continuous function $g$ (in terms of the left derivative) then
we may assume that $\frac{\partial\varphi_{\epsilon}(t)}{\partial t}\leq g.$ \end{lem}
\begin{proof}
It follows immediately from its definition that the projection operator
$P_{\omega}$ is concave and in particular locally Lip continuous
as a function of $t.$ As for the approximation property it seems
likely that it can be deduced in a much more general setting from
an appropriate parametrized version of the approximation schemes for
$\omega-$psh function introduced by Demailly. But here we note that
a direct proof can be given exploiting that $dd^{c}\phi(t)$ is in
$L^{\infty}$ and in particular $\phi(t)$ is in $\mathcal{C}^{1}(X).$
Indeed, $\varphi_{\epsilon}(t)$ can be defined by using local convolutions
(which gives local $C^{1}-$convergence) together with a partition
of unity and finally replacing $\phi_{\epsilon}(t)$ with $(1-\delta_{1}(\epsilon))\phi_{\epsilon}(t)-\delta_{2}(\epsilon)t$
for appropriate sequence $\delta_{i}(\epsilon)$ tending to zero with
$\epsilon.$ The point is that, by the $C^{1}-$convergence the error
terms coming from the first derivatives on the partition of unity
are negligible and hence $\phi_{\epsilon}(t)$ is $\omega-$psh up
to a term of order $o(\epsilon).$ Indeed, setting 
\[
\phi_{\epsilon}(t):=\sum_{i=1}^{m}\rho_{i}\phi_{\epsilon}^{(i)}(t),\,\,\,\,1=\sum_{i=1}^{m}\rho_{i},\,\,\,\rho_{i}\in C_{c}^{\infty}(X)
\]
and using Leibniz rule gives $dd^{c}\phi_{\epsilon}(t)=\sum_{i=1}^{m}(\rho_{i}dd^{c}\phi_{\epsilon}^{(i)}(t)+R_{i}(\phi_{\epsilon}^{(i)}))$
where the second term $R(\phi_{\epsilon})$ only depends on the first
order jet of $\phi.$ Now, by the local $C^{1}-$convergence $R(\phi_{\epsilon})=R(\phi)+o(\epsilon).$
But $R(\phi)$ vanishes (since $\phi=\rho_{1}\phi+...$ and $dd^{c}\phi=\rho_{1}dd^{c}\phi+$...)
and hence $dd^{c}\phi_{\epsilon}(t)=\sum_{i=1}^{m}\rho_{i}dd^{c}\phi_{\epsilon}^{(i)}(t)+o(\epsilon).$
Finally, from the definition of convolution we have $dd^{c}\phi_{\epsilon}(t)+\omega\geq-C\epsilon\omega$
and $\frac{\partial\varphi_{\epsilon}(t)}{\partial t}\leq g+C\epsilon$
(for some positive constant $C)$ and hence we may first take $\delta_{1}(\epsilon)=C\epsilon$
and then $\delta_{2}(\epsilon)=C\epsilon(1+\sup|g|).$\end{proof}
\begin{rem}
The parametrized non-coincidence sets $\Omega_{t}:=\{P_{\omega}(\varphi+tf)<(\varphi+tf)$
are, in fact, increasing in $t.$ Indeed, as shown in the proof of
Proposition \ref{prop:The-sup-coincides with projection of linear curve}
$P_{\omega}(\varphi+tf)-(\varphi+tf)$ is decreasing in $t.$
\end{rem}
We will also have use for the following generalized envelope associated
to a given compact subset $K$ of a Kähler manifold $(X,\omega)$
and a lsc function $f$ on $X:$ 
\[
P_{(K,\omega)}(f)(x):=\sup_{u\in PSH(X,\omega)}\{u(x):\,\,\,u\leq f\,\mbox{on\,\ensuremath{K}}\}
\]
(the function $V_{K,\omega}:=P_{(K,\omega)}(0)$ is called the global
extremal function of $(K,\omega)$ in \cite{gz} ). 

We recall that a subset $K$ in $X$ is said to be \emph{non-pluripolar}
if it is not locally contained in the $-\infty-$set of a local psh
function. 
\begin{lem}
\label{lem:proj of plus infity} Suppose that $f$ is a lsc function
on a compact Kähler manifold $(X,\omega)$ taking values in $]0,\infty]$
such that $f$ is bounded from above on $K,$ where $K$ is non-pluripolar.
Then the function $P_{(K,\omega)}(f)$ is bounded from above. As a
consequence, if $X=K$ and $f$ is locally bounded on the complement
of an analytic subvariety, then $P_{(X,\omega)}(f)$ is bounded from
above.\end{lem}
\begin{proof}
By assumption $P_{(K,\omega)}(f)\leq P_{(K,\omega)}(0)+C$ for $C$
a sufficiently large constant. But it is well-known that $P_{(K,\omega)}(0)$
is finite iff $K$ is non-pluripolar \cite{gz}. The last statement
of the lemma then follows by fixing a coordinate ball $B$ contained
in the open subset where $f$ is locally bounded and using that $P_{(X,\omega)}(f)\leq P_{(B,\omega)}(f)<\infty.$ 
\end{proof}
In general, $P_{(K,\omega)}(f)$ is not upper semicontinuous. But
we recall that $K$ is said to be \emph{regular} (in the sense of
pluripotential theory) if $P_{(K,\omega)}(f)$ is continuous (and
hence $\omega-$psh) for any continuos function $f$ (see\cite{bbw}
and references therein).

\subsection{\label{sub:A-priori-estimates}A priori estimates }

The key element in the proof of Theorem \ref{thm:main general} is
the Laplacian estimate which provides a uniform bound on the metrics
$\omega^{(\beta)}(t)$ on any fixed time interval. There are various
well-known approaches for providing such an estimate for a fixed $\beta,$
using parabolic versions of the classical estimate of Aubin and Yau
and its variants. However, in our setting one has to make sure that
all the estimates are uniform in $\beta$ and that they do not rely
on a uniform positive lower bound on $\omega^{(\beta)}(t)$ (which
is not available).

\subsubsection{The Laplacian estimate in the one dimensional case}

We start with the one-dimensional case where the Laplacian estimate
becomes particularly explicit:
\begin{prop}
When $n=1$ we have, for any fixed Kähler form $\omega$ on $X$ 
\[
\left\Vert \omega^{(\beta)}(t)\right\Vert \leq\max\left\{ \left\Vert \omega_{0}\right\Vert ,\left\Vert \theta_{\beta}-\frac{\mbox{Ric }\omega}{\beta})\right\Vert \right\} 
\]
 in term of the sup norm defined by $\omega$ (i.e. $\left\Vert \eta\right\Vert :=\sup_{X}\left|\eta/\omega\right|).$\end{prop}
\begin{proof}
We write the normalized KRF as 
\[
h:=\log\frac{\omega_{\varphi_{t}}}{\omega}=\beta\left(\varphi_{t}-f_{\beta}+\frac{\partial\varphi_{t}}{\partial t}\right).
\]
 Applying the parabolic operator $\frac{1}{\beta}\Delta_{t}-\frac{\partial}{\partial t},$
where $\Delta_{t}(=dd^{c}/\omega_{\varphi_{t}})$ denotes the Laplacian
wrt the metric $\omega_{\varphi_{t}},$ to the equation above gives
\[
\frac{1}{\beta}\Delta_{t}h-\frac{\partial}{\partial t}h=\frac{1}{\omega_{\varphi_{t}}}(dd^{c}\varphi_{t}-dd^{c}f_{\beta})+\Delta_{t}\frac{\partial\varphi_{t}}{\partial t}-\frac{\partial}{\partial t}h.
\]
Now
\[
\frac{\partial}{\partial t}h:=\frac{\partial}{\partial t}\log\left(\frac{e^{-t}\omega_{0}+(1-e^{-t})\chi_{\beta}+dd^{c}\varphi_{t}}{\omega}\right)=\frac{1}{\omega_{\varphi_{t}}}\left(-e^{-t}\omega_{0}+e^{-t}\chi_{\beta}+dd^{c}\frac{\partial\varphi_{t}}{\partial t}\right).
\]
 Hence, the two terms involving $\frac{\partial\varphi_{t}}{\partial t}$
cancel, giving 

\begin{equation}
\frac{1}{\beta}\Delta_{t}h-\frac{\partial}{\partial t}h=\frac{\omega}{\omega_{\varphi_{t}}}\left(\Delta_{\omega}(\varphi_{t}-f_{\beta}))-e^{-t}(\chi_{\beta}-\omega_{0})/\omega\right),\label{eq:pf of lapl estimate n is one}
\end{equation}
 i.e. 
\[
\omega_{\varphi_{t}}\left(\frac{1}{\beta}\Delta_{t}h-\frac{\partial}{\partial t}h\right)+dd^{c}f_{\beta}+e^{-t}(\chi_{\beta}-\omega_{0})=dd^{c}\varphi_{t},
\]
which in terms of $\omega(t)\;\left(:=\omega_{0}+(1-e^{-t})(\chi_{\beta}-\omega_{0})+dd^{c}\varphi_{t}\right)$
becomes 
\[
\omega_{\varphi_{t}}\left(\frac{1}{\beta}\Delta_{t}h-\frac{\partial}{\partial t}h\right)+dd^{c}f_{\beta}+\chi_{\beta}=\omega(t).
\]
Applying the parabolic maximum principle to $h$ concludes the proof.
Indeed, there are two alternatives: either $h$ has its maximum on
$X\times[0,T]$ (for $T$ fixed) at $t=0$ which implies that $\mbox{Tr}_{\omega}\omega(t)\leq\mbox{Tr}_{\omega}\omega_{0}$
on $X\times[0,T],$ or the maximum of $h$ is attained at a point
$(x,t)$ in $X\times]0,T].$ In the latter case $\mbox{Tr}_{\omega}\omega(t)\leq\sup_{X}\mbox{Tr}_{\omega}(dd^{c}f_{\beta}+\chi_{\beta})\leq C$
(since $dd^{c}f_{\beta}+\chi_{\beta}=\theta_{\beta}-\frac{1}{\beta}\mbox{Ric \ensuremath{\omega=\theta+o(1)).}}$
\end{proof}

\subsubsection{The upper bound on $\varphi_{t}$\label{sub:The-upper-bound}}

Next, we come back to the general case. Writing the normalized KRF
flow as 
\begin{equation}
-\frac{\partial(\varphi_{t}-f_{\beta})}{\partial t}+\frac{1}{\beta}\log\frac{(\hat{\omega}_{t}+dd^{c}f_{\beta}+dd^{c}(\varphi_{t}-f_{\beta}))^{n}}{\omega^{n}}=\varphi_{t}-f_{\beta},\label{eq:KRF in proof of upper bound of phi}
\end{equation}
it follows immediately from the parabolic maximum principle that 
\[
\varphi_{t}(x)-f_{\beta}(x)\leq\max\left\{ \sup_{X}(0-f_{\beta}),\frac{1}{\beta}\sup_{X}\log\frac{(\hat{\omega}_{t}+dd^{c}f_{\beta}+0)^{n}}{\omega^{n}}\right\} \leq A/\beta,
\]
where $A$ only depends on the upper bounds of $\theta_{\beta}$.
In particular, 
\[
\varphi_{t}\leq P_{t}(f_{\beta})+A/\beta
\]
 and, as a consequence, 
\[
\leq P_{C'\omega}(f_{\beta})+A/\beta
\]
where $C'$ is any constant satisfying $\chi_{\beta}\leq C'\omega$
and $\omega_{0}\leq C'\omega$ (thus ensuring that $PSH(X,\hat{\omega}_{t})\subset PSH(X,C'\omega)).$

\subsubsection{\label{sub:The-lower-bounds on deriv}The lower bounds on $\frac{\partial\tilde{\varphi}_{s}}{\partial s}$
and $\frac{\partial\varphi_{t}}{\partial t}$}

Differentiating the non-normalized KRF with respect to $s$ gives,
with $g(x,s):=-\frac{\partial\tilde{\varphi}_{s}(x)}{\partial s},$
\[
\frac{\partial g}{\partial s}-\frac{1}{\beta}\Delta_{s}g=-\frac{1}{\beta}\mbox{\ensuremath{\mathrm{Tr_{s}}} }(\chi_{\beta})\leq0.
\]
Hence, by the parabolic maximum principle the sup of $g$ is attained
at $t=0$ which gives

\begin{equation}
-\frac{\partial\tilde{\varphi}_{s}}{\partial s}\leq C_{1},\,\,\,\,C_{1}=\sup_{X}\left(-\frac{1}{\beta}\log\frac{\omega_{\varphi_{0}}^{n}}{\omega^{n}}-f_{\beta}\right)\label{eq:lower bound on time deriv}
\end{equation}
where $C_{1}$ thus only depends on the strict positive lower bound
of $\omega_{\varphi_{0}}^{n}$ and on $\inf_{X}(f_{\beta})$ (which
by our normalizations vanishes). 

Next, using that 
\begin{equation}
\frac{\partial\tilde{\varphi}_{s}}{\partial s}=\frac{\partial\varphi_{t}}{\partial t}+\varphi_{t}+nt/\beta\label{eq:relation between t deriv and s deriv}
\end{equation}
gives 
\begin{equation}
\frac{\partial\varphi_{t}}{\partial t}\geq-C_{1}-\varphi_{t}-nt/\beta\geq-C_{1}'-nt/\beta\label{eq:lower bound time derivative normalized}
\end{equation}
using the previous upper bound on $\varphi_{t}.$

\subsubsection{The lower bound on $\varphi_{t}$ }

It follows immediately from the previous bound that 
\[
\varphi_{t}\geq\varphi_{0}-C_{1}'t-nt^{2}/2\beta.
\]

\subsubsection{The Laplacian bound }

We will use Siu's well-known variant \cite[pp. 98-99]{siu} of the
classical Aubin-Yau Laplacian estimate 
\begin{lem}
\label{lem:Siu}Given two Kähler forms $\omega'$ and $\omega$ such
that $\omega'^{n}=e^{F}\omega^{n}$ we have that 
\[
\Delta_{\omega'}\log\mathrm{Tr_{\omega}}\omega'\geq\frac{\mathrm{Tr}_{\omega}dd^{c}F}{\mathrm{Tr_{\omega}}\omega'}-B_{+}\mathrm{Tr}_{\omega'}\omega,
\]
 where the constant $B_{+}$ is a multiple of the absolute value of
the infimum on $X$ of the holomorphic bisectional curvatures of $\omega.$\end{lem}
\begin{proof}
In the original statement in \cite[pp. 98-99]{siu} it was assumed
that $\omega'$ and $\omega$ are cohomologous, but since the proof
is local this assumption is not needed. See for example \cite[Prop 4.1.2]{b-g}
where it is shown that 
\[
\Delta_{\omega'}\log\mathrm{Tr_{\omega}}\omega'\geq\frac{\mathrm{Tr}_{\omega}(-\mbox{Ric \ensuremath{\omega'}})}{\mathrm{Tr_{\omega}}\omega'}+B\mathrm{Tr}_{\omega'}\omega,
\]
where $B$ is the infimum of the holomorphic bisectional curvatures
of $\omega.$ In our notations, $-\mbox{Ric \ensuremath{\omega'}}=dd^{c}F\text{-\mbox{Ric \ensuremath{\omega}}}$
and since $\mathrm{-Tr}_{\omega}\mbox{Ric \ensuremath{\omega\geq-c_{n}|B|}}$
and $\mathrm{Tr_{\omega}}\omega'\mathrm{Tr_{\omega'}}\omega\geq n$
we arrive at the inequality in the statement of the lemma.
\end{proof}
We start with the case when $X$ admits a Kähler metric $\omega$
with non-negative holomorphic bisectional curvature. In this case
the constant $B_{+}$ vanishes.
\begin{prop}
Suppose that $X$ admits a Kähler metric $\omega$ with non-negative
bisectional curvature. Then 
\[
\left\Vert \omega^{(\beta)}(t)\right\Vert \leq\max\left\{ \left\Vert \omega_{0}\right\Vert ,\left\Vert \theta_{\beta}-\frac{\mbox{Ric }\omega}{\beta}\right\Vert \right\} .
\]
\end{prop}
\begin{proof}
Setting 
\[
h:=\log\mathrm{Tr}_{\omega}\omega',
\]
where $\omega'=\hat{\omega}_{t}+dd^{c}\varphi(t)$, we get, using
Siu's inequality, 
\[
\left(-\frac{\partial}{\partial t}h+\frac{1}{\beta}\Delta_{t}h\right)\geq\frac{1}{\mathrm{Tr}_{\omega}\omega'}\Delta_{\omega}(\varphi_{t}-f_{\beta})+\frac{1}{\mathrm{Tr}_{\omega}\omega'}\left(\Delta_{\omega}\frac{\partial\varphi_{t}}{\partial t}\right)-\frac{\partial}{\partial t}h.
\]
The rest of the proof then proceeds precisely as in the Riemann surface
case. 
\end{proof}
In the general case we get the following 
\begin{prop}
\label{prop:general Laplac est}There is a constant $C$ such that,
for $\beta>\beta_{0}$

\[
\omega^{(\beta)}(t)\leq e^{C(1+1/\beta)(1+t)e^{t}}\omega,
\]
 where $C$ depends on the same quantities as in the statement of
Theorem \ref{thm:main general}.\end{prop}
\begin{proof}
Recall that by abuse of notation we set $\omega_{t}=\hat{\omega}_{t}+dd^{c}\varphi^{(\beta)}(t)$.
By the Laplacian inequality (Lemma \ref{lem:Siu}) we have 

\begin{multline*}
\frac{1}{\beta}B_{+}{\rm Tr}_{\omega_{t}}\omega+\left(-\frac{\partial}{\partial t}\log{\rm Tr}_{\omega}\omega_{t}+\frac{1}{\beta}\Delta_{t}\log\mathrm{Tr}_{\omega}\omega_{t}\right)\\
\geq\frac{\Delta_{\omega}(\varphi_{t}-f_{\beta})-e^{-t}{\rm Tr}_{\omega}(\chi_{\beta}-\omega_{0})}{\mathrm{Tr}_{\omega_{t}}\omega'}
\end{multline*}
thanks to the cancelation of the terms involving $\partial\varphi_{t}/\partial t$,
just as before. To handle the first term in the left-hand side above
we note that 
\[
\omega\leq Ce^{t}\hat{\omega}_{t},
\]
 where $1/C$ is a positive lower bound for $\omega_{0}.$ Since $\text{\ensuremath{\mathrm{Tr}}}_{\omega_{t}}\hat{\omega}_{t}=n-\Delta_{\omega_{t}}\varphi$
we thus get, by setting 
\[
G(x,t):=\log\left(\text{\ensuremath{\mathrm{Tr}}}_{\omega}\omega_{t}\right)-B_{+}Ce^{t}\varphi_{t}-f(t),
\]
 for any given function $f(t)$ of $t,$ 
\begin{multline*}
-\frac{\partial f(t)}{\partial t}-CB_{+}\frac{\partial(e^{t}\varphi_{t})}{\partial t}+\frac{n}{\beta}B_{+}Ce^{t}+\left(-\frac{\partial}{\partial t}G+\frac{1}{\beta}\Delta_{t}G\right)\\
\geq\frac{\Delta_{\omega}(\varphi_{t}-f_{\beta})-e^{-t}{\rm Tr}_{\omega}(\chi_{\beta}-\omega_{0})}{{\rm Tr}_{\omega}\omega_{t}}.
\end{multline*}
Next we note that, thanks to the lower bound on $\frac{\partial\varphi_{t}}{\partial t}$
above \ref{eq:lower bound time derivative normalized} we have 
\[
\frac{\partial(e^{t}\varphi_{t})}{\partial t}\geq-e^{t}\left(C+\frac{nt}{\beta}\right).
\]
Hence, taking $f(t)=C'(1+t)(1+1/\beta)e^{t}$ for $C'$ sufficiently
large gives 
\[
\left(-\frac{\partial}{\partial t}+\frac{1}{\beta}\Delta_{t}\right)G\geq\frac{\text{\ensuremath{\mathrm{Tr}}}_{\omega}\left(dd^{c}\varphi_{t}-dd^{c}f_{\beta}-e^{-t}(\chi_{\beta}-\omega_{0})\right)}{\text{\ensuremath{\mathrm{Tr}}}_{\omega}\omega_{t}}.
\]
Since $e^{-t}(\chi_{\beta}-\omega_{0})=\chi_{\beta}-\hat{\omega}_{t}$
this implies that
\[
\left(-\frac{\partial}{\partial t}+\frac{1}{\beta}\Delta_{t}\right)G\geq\frac{\text{Tr}_{\omega}(\omega_{t}-dd^{c}f_{\beta}-\chi_{\beta})}{\mathrm{Tr}_{\omega}\omega_{t}}.
\]
Finally, using $\chi_{\beta}+dd^{c}f_{\beta}=\theta_{\beta}-\frac{1}{\beta}\mbox{Ric \ensuremath{\omega}}$
this shows that the estimate on $\text{\ensuremath{\mathrm{Tr}}}_{\omega}\omega_{t}$
we get from the parabolic maximum principle applied to $G$ only depends
on $\chi_{\beta}$ through the upper bound on $\varphi_{t}$ (which
in turn depends on an upper bound on $\chi_{\beta}$ and is of the
order $1/\beta).$ 
\end{proof}

\subsubsection{The upper bound on $\frac{\partial\varphi_{t}}{\partial t}$ and
$\frac{\partial\tilde{\varphi}_{s}}{\partial s}$}

(these upper bounds are not needed for the proof of the convergence
in Theorem \ref{thm:main intro}). From the upper bound on $\omega(t)$
and the defining equations for the KRFs one directly obtains bounds
on $\frac{\partial\varphi_{t}}{\partial t}$ and $\frac{\partial\tilde{\varphi}_{s}}{\partial s}.$
However, better bounds can be obtained by a variant of the proof of
the lower bounds on $\frac{\partial\varphi_{t}}{\partial t}$ and
$\frac{\partial\tilde{\varphi}_{s}}{\partial s}.$ Indeed, differentiating
the normalized and the non-normalized KRFs, respectively gives 
\begin{equation}
\frac{\partial\frac{\partial\tilde{\varphi}_{s}}{\partial s}}{\partial s}-\Delta_{s}\frac{\partial\frac{\partial\tilde{\varphi}_{s}}{\partial s}}{\partial s}-\mbox{Tr}_{s}\chi_{\beta}=0\label{eq:linearized NN KRF}
\end{equation}
and 
\begin{equation}
\frac{\partial(e^{t}\frac{\partial\varphi_{t}}{\partial t})}{\partial t}-\Delta_{t}\frac{\partial(e^{t}\frac{\partial\varphi_{t}}{\partial t})}{\partial t}-\mbox{Tr}_{t}(\chi_{\beta}-\omega_{0})=0.\label{eq:linearized N KRF}
\end{equation}
Using that $\omega(s)=e^{t}\omega(t),$ $ds/d=e^{-t}dt/d$ and $\frac{\partial\tilde{\varphi}_{s}}{\partial s}=\frac{\partial\varphi_{t}}{\partial t}+\varphi_{t}+nt/\beta$
the first equation above becomes
\[
\frac{\partial(\frac{\partial\varphi_{t}}{\partial t}+\varphi_{t}+nt/\beta)}{\partial t}-\Delta_{t}\frac{\partial(\varphi_{t}+nt/\beta)}{\partial t}-\mbox{Tr}_{t}\chi_{\beta}=0.
\]
Hence, taking the differences between equations \ref{eq:linearized NN KRF}
and \ref{eq:linearized N KRF} gives that $g:=e^{t}\frac{\partial\varphi_{t}}{\partial t}-\frac{\partial\varphi_{t}}{\partial t}-\varphi_{t}-nt/\beta$
satisfies
\[
\frac{\partial g}{\partial t}-\Delta_{t}\frac{\partial g}{\partial t}=-\mbox{Tr }_{t}\omega_{0}\leq0.
\]
 Accordingly, the parabolic maximum principle reveals that the sup
over $X$ of $e^{t}\frac{\partial\varphi_{t}}{\partial t}-\frac{\partial\varphi_{t}}{\partial t}-\varphi_{t}-nt/\beta$
is decreasing, thanks to the upper bound on $\varphi_{t},$ 
\[
\frac{\partial\varphi_{t}}{\partial t}\leq\frac{\sup_{X}P_{t}(f_{\beta})+(A+nt)/\beta)}{(e^{t}-1)}
\]
 (this is a minor generalization of the estimate in \cite{t-z}).
Finally, this yields 
\[
\frac{\partial\tilde{\varphi}_{s}}{\partial s}\leq C''\frac{1+\beta^{-1}\log(1+s)}{s}.
\]

\subsection{\label{sub:Existence-and-characterizations}Existence and characterizations
of the large $\beta$ limit of the KRF}

By the previous estimates there is a subsequence of $\varphi^{(\beta)}(t)$
which converges uniformly (and even in $C^{1,\alpha}-$norm) to a
limiting Lip curve $\varphi(t)$ with values in $PSH(X,\hat{\omega}_{t}).$
As we next show $\varphi(t)$ is uniquely determined, i.e. the whole
family converges to $\varphi(t).$
\begin{prop}
\label{prop:existence of the limit}The large $\beta-$limit of $\varphi^{(\beta)}(t)$
of the normalized KRF exists: it is equal to the curve defined as
the sup over all curves $\psi(t)$ in $PSH(X,\omega)$ such that $\psi(0)=\varphi(0)$
and such that $\varphi^{(\beta)}(t)$ is locally Lipchitz in $t$
(for $t>0)$ and in $C^{1}(X),$ for a fixed $t$ and \textup{
\[
\frac{\partial\psi(t)}{\partial t}\leq-\psi(t)+f
\]
 (in the weak sense), or equivalently: such that} \textup{
\[
(\psi(t)-f)e^{t}
\]
 is decreasing in time.}\end{prop}
\begin{proof}
By the second order a priori estimates we have 
\[
\frac{d\varphi^{(\beta)}(t)}{dt}\leq\frac{C(t)}{\beta}-\varphi^{(\beta)}(t)+f_{\beta}
\]
and hence the limiting Lip curve $\varphi(t)$ satisfies
\[
\frac{d\varphi(t)}{dt}\leq-\varphi(t)+f
\]
in the weak sense, i.e. $\varphi(t)$ is a candidate for the sup appearing
in the statement of the proposition. Alternatively, we get 
\[
\frac{d\left((\varphi^{(\beta)}-f_{\beta})(e^{t}-\frac{C}{\beta}e^{t})\right)}{dt}\leq0,
\]
 i.e. 
\[
(\varphi^{(\beta)}-f_{\text{\ensuremath{\beta}}})(t)\leq\frac{1}{(1-C/\beta)}e^{-t}g_{\beta}(x,t),
\]
 where $g_{\beta}(x,t)$ is decreasing in time. Hence, after passing
to a subsequence the limit satisfies 
\[
(\varphi-f)(t)\leq e^{-t}g(x,t),
\]
 where $g(x,t)$ is decreasing in time.

Next, by the parabolic maximum principle $\varphi^{(\beta)}(t)$ is
the sup over all smooth curves $u_{\beta}(t)$ with values in (the
interior of) $PSH(X,\hat{\omega}_{t})$ such that $u_{\beta}(0)=\varphi(0)$
and 
\[
\frac{du_{\beta}}{dt}\leq\frac{1}{\beta}\log\frac{(\hat{\omega}_{t}+dd^{c}u_{\beta}(t))^{n}}{\omega^{n}}-(u_{\beta}(t)-f_{\beta})
\]
on a fixed time-interval $[0,T].$ Now take a smooth curve $v(t)$
from $[0,T]$ to $PSH(X,\hat{\omega}_{t})\cap C^{\infty}(X)$ such
that and $v(0)=\varphi(0)$ and such that 
\[
\frac{d}{dt}v(t)\leq-(v(t)-f).
\]
We set 
\[
v_{\epsilon}(t):=(1-\epsilon)v(t)-\epsilon,
\]
ensuring that 
\[
\frac{d}{dt}v_{\epsilon}(t)\leq-v_{\epsilon}(t)-f-\epsilon
\]
 and 
\[
(\hat{\omega}_{t}+dd^{c}v_{\epsilon}(t))^{n}\geq\epsilon^{n}\hat{\omega}_{t}^{n}\geq\epsilon^{n}C(T)\omega^{n}.
\]
Hence, for $\beta$ sufficiently large (depending on the lower bound
$C(T)$ of the positivity of $\hat{\omega}_{t}^{n}$ on $[0,T]$ and
the convergence speed of $f_{\beta}$ towards $f$), 
\[
\frac{dv_{\epsilon}(t)}{dt}\leq\frac{1}{\beta}\log\frac{(\hat{\omega}_{t}+dd^{c}v_{\epsilon}(t))^{n}}{\omega^{n}}-v_{\epsilon}(t)-f_{\text{\ensuremath{\beta}}}.
\]
But then it follows from the parabolic maximum principle that $v_{\epsilon}(t)\leq\varphi_{\beta,\epsilon}(t),$
for $\beta>>1,$ where $\varphi_{\beta,\epsilon}(t)$ satisfies the
same KRF as $\varphi^{(\beta)}(t)$, but with initial value $(1-\epsilon)\varphi_{0}+\epsilon v_{0}.$
By the maximum principle we have 
\[
\left|\varphi_{\beta,\epsilon}(t)-\varphi^{(\beta)}(t)\right|\leq C\epsilon
\]
and hence letting $\beta\rightarrow\infty$ gives, for any limit $\varphi(t)$
of $\varphi^{(\beta)}(t)$ 
\[
v_{\epsilon}(t)\leq\varphi(t)+C\epsilon.
\]
Since $\epsilon$ was arbitrary this gives $v(t)\leq\varphi(t).$
All that remain is thus to show that the smoothness assumption on
$v(t)$ can be removed. This could be done by working with the notion
of viscosity subsolutions \cite{egz}, but here we will use a more
direct approach by first noting that the sup above is realized by
$P_{\hat{\omega}_{t}}(e^{-t}\varphi_{0}+(1-e^{-t})f),$ as shown in
the next proposition. Then we can use the regularization in Lemma
\ref{lem:concave proj} together with the slight generalization of
the parabolic comparison principle formulated in Remark \ref{rem:comp for concave})
to conclude. \end{proof}
\begin{prop}
\label{prop:The-sup-coincides with projection of linear curve}The
sup in the previous proposition coincides with $P_{\hat{\omega}_{t}}(e^{-t}\varphi_{0}+(1-e^{-t})f).$\end{prop}
\begin{proof}
It will be convenient to use the equivalent ``non-normalized setting''
which means that we replace the convex combination above with $\varphi_{0}+tf$
and have to prove that $a(t):=P_{t}(\varphi_{0}+tf)-tf$, where $P_{t}=P_{\omega+t\theta}$,
is decreasing, i.e. that $a(t+s)-a(t)\leq0$ for any fixed $t,s\geq0$
(compare Remark \ref{rem:scaling of linear evolution}). To this end
we rewrite the difference above as 
\[
P_{t+s}(\varphi_{0}+tf+sf)-P_{t}(\varphi_{0}+tf)-sf=P_{t+s}((1-\lambda)\varphi_{0}+\lambda\psi_{t})-P_{t}(\psi_{t})-sf,
\]
 where 
\[
\psi_{t}:=\varphi_{0}+tf,\,\,\,\lambda:=(t+s)/t.
\]
In particular, $\lambda\geq1$ and hence it follows from the very
definition of $P$ (as an upper envelope wrt a convex set) that
\[
P_{t+s}((1-\lambda)\varphi_{0}+\lambda\psi_{t})\leq(1-\lambda)\varphi_{0}+\lambda P_{t}(\psi_{t}),
\]
 which gives that $a(t+s)-a(t)$ can be estimated from above by 
\[
(1-\lambda)\varphi_{0}+\lambda P_{t}(\psi_{t})-P_{t}(\psi_{t})-sf=(\lambda-1)tf-sf=0
\]
as desired (a similar direct proof can be given for the normalized
KRF, but using instead $\lambda=(1-e^{-(t+s)})/(1-e^{-t})$). \end{proof}
\begin{rem}
It is possible to prove the uniform large $\beta$-convergence of
the flows $\varphi^{(\beta)}(t)$ directly without the Laplacian estimate
and without going through the characterization in terms of curves
appearing in \ref{prop:existence of the limit}. Indeed, an upper
bound of the form $\varphi^{(\beta)}(t)\leq P_{t}(\varphi_{0}+tf_{\beta})+C_{t}/\beta$
can be proved (in the non-normalized setting) by applying the parabolic
maximum principle. Indeed, for some uniform constant $C$ it can be
shown that the function 
\[
\varphi_{0}+tf_{\beta}+\frac{Ct+nt\log(t+1)}{\beta}
\]
is a super-solution of the parabolic complex Monge-Ampère equation.
The lower bound is then proved as before using regularization and
the parabolic comparison principle. Alternatively, the lower bound
can also be proved directly without regularization using the monotonicity
property (Proposition \ref{prop:The-sup-coincides with projection of linear curve})
and the viscosity comparison principle to get
\[
\varphi^{(\beta)}(t)\geq(1-\delta)P_{t}(\varphi_{0}+tf_{\beta})-\frac{C_{t}(1-\log(1-\delta))}{\beta}-(1-\delta)C_{t},\ \delta\in(0,1).
\]
More generally, the uniform convergence holds even if $f$ is merely
continuous (see Section \ref{sub:An-extension-of}). 
\end{rem}

\section{\label{sec:Large-time-asymptotics}Large time asymptotics of the
flows}

In order to study the joint large $t$ and large $\beta-$limit of
the non-normalized Kähler-Ricci flows $\omega^{(\beta)}(t)$ introduced
in the previous section we consider, as usual, the normalized Kähler
forms $\omega^{(\beta)}(t)/(t+1)$ (which have uniformly bounded volume)
evolving according to the normalized Kähler-Ricci flow \ref{eq:normalized krf in setup}.
Our first observation is that the following double limit always exists:

\begin{equation}
\lim_{t\rightarrow\infty}\lim_{\beta\rightarrow\infty}\omega^{(\beta)}(t)/(t+1)=P(\theta),\label{eq:double scaling limit}
\end{equation}
for any initial Kähler metric $\omega_{0}$ (where the large $t-$limit
holds in the weak topology of currents). This follows immediately
from Theorem \ref{thm:main general} combined with the following
\begin{lem}
\label{lem:conv of envelopes}Assume that $\chi\geq0$ and set $\hat{\omega}_{t}:=e^{-t}\omega_{0}+(1-e^{-t})\chi$.
Then, for any given smooth functions $\varphi_{0}$ and $f$ on $X$,
\[
P_{\hat{\omega}_{t}}(e^{-t}\varphi_{0}+(1-e^{-t})f)\rightarrow P_{\chi}f
\]
 as $t\rightarrow\infty,$ in the $L^{1}$-topology. In particular,
if $[\theta]\geq0$ then $P(e^{-t}\omega_{0}+(1-e^{-t})\theta)\rightarrow P(\theta)$
in the weak topology of currents.\end{lem}
\begin{proof}
Set $\psi_{t}:=P_{\hat{\omega}_{t}}(e^{-t}\varphi_{0}+(1-e^{-t})f)=:P_{\hat{\omega}_{t}}(f(t)).$
Since $\psi_{t}\in PSH(X,C\omega)$ for $C$ sufficiently large the
family $\psi_{t}$ is relatively compact in the $L^{1}$-topology.
We denote by $\psi_{\infty}$ a given limit point of $\psi_{t},$
which clearly is in $PSH(X,\chi).$ Moreover, $\psi_{t}\leq f(t)$
implies $\psi_{\infty}\leq f$ and hence $\psi_{\infty}\leq P_{\chi}f.$
To prove the converse we set $\psi:=P_{\chi}f$ and fix $\delta>0.$
Observe that $dd^{c}(1-\delta)\psi+\hat{\omega}_{t}\geq(1-\delta)dd^{c}\psi+(1-e^{-t})\chi\geq(1-\delta)(dd^{c}\psi+\chi)\geq0,$
for $t>>1.$ Hence, since $\psi$ is bounded we get $(1-\delta)\psi\leq P_{\hat{\omega}_{t}}(f+C\delta)\leq P_{\hat{\omega}_{t}}(f_{t})+C\delta+C'e^{-t}.$
Hence, letting first $t\rightarrow\infty$ gives $(1-\delta)\psi\leq\psi_{\infty}+C\delta.$
Finally, letting $\delta\rightarrow0$ concludes the proof.
\end{proof}
In the following two sections we will look closer at the situation
appearing in the two extreme cases, where $\frac{\text{1}}{\beta}c_{1}(K_{X})+[\theta_{\beta}]$
is positive and trivial, respectively. Then we will make some comments
on the intermediate cases and the relations to previous results in
complex geometry concerning the case when $\beta$ is fixed.

\subsection{\label{sub:The-case-when convex envelopes}The case when $c_{1}(K_{X})/\beta+[\theta_{\beta}]=[\omega_{0}]:$
a dynamic construction of  envelopes }

In this section we will consider the situation when the normalized
KRF preserves the initial cohomology class. Given a volume form $dV$
on $X$ and a smooth function $f$ on $X,$ setting $\theta=dd^{c}f+\omega$
and 
\[
\theta_{\beta}=\theta+\frac{1}{\beta}\mbox{Ric }dV
\]
 for a fixed choice of Kähler metric $\omega\in[\theta]$ the normalized
KRF in $[\theta]$ on the level of Kähler potentials becomes

\begin{equation}
\frac{\partial\varphi^{(\beta)}(t)}{\partial t}=\frac{1}{\beta}\log\frac{(\omega+dd^{c}\varphi^{(\beta)}(t))^{n}}{dV}-\varphi^{(\beta)}(t)+f\label{eq:ma flow for envelope}
\end{equation}
(hence $f_{\beta}=f+\frac{1}{\beta}\log(dV/\omega^{n})$ and the reference
Kähler metric on $X$ and $\chi_{\beta}=\chi$ in $[\theta]$ are
both taken as $\omega$ in this setting). 
\begin{thm}
\label{thm:envelope from dynamics}Let $(X,\omega)$ be a Kähler manifold
and fix a volume form $dV$ on $X.$ Given a smooth function $f$
we denote by $\varphi^{(\beta)}(x,t)$ the solution of the evolution
equation \ref{eq:ma flow for envelope}\textup{ with initial data
$\varphi_{0}$ and set $\varphi_{t}^{(\infty)}:=P_{\omega}(e^{-t}\varphi_{0}+(1-e^{-t})f).$
Then 
\begin{equation}
\sup_{X}\left|\varphi_{t}^{(\beta)}-\varphi_{t}^{(\infty)}\right|\leq C\frac{\log\beta}{\beta}\label{eq:rate of conv in theorem envelope from dyn}
\end{equation}
}

\textup{and there is a constant $C$ such that 
\begin{equation}
\left|\frac{\partial\varphi}{\partial t}\right|\leq Ce^{-t},\,\,\,|dd^{c}\varphi_{t}|_{\omega_{0}}\leq C\label{eq:bounds in theorem env from dyn}
\end{equation}
}\end{thm}
\begin{proof}
Note that $g:=e^{t}\frac{\partial\varphi_{t}}{\partial t}$ satisfies
\[
\frac{\partial g}{\partial t}-\Delta_{t}g=0
\]
 and hence, by the parabolic maximum principle, $|g(x,t)|\leq\sup_{X}|g(x,0):=C,$
i.e $|\frac{\partial\varphi_{t}^{(\beta)}}{\partial t}|\leq Ce^{-t}.$
But then we can (for $t$ large) employ the function $f(t)=-Ce^{-t}$
in the proof of the Laplacian estimate in Prop \ref{prop:general Laplac est}
(since the cohomological term vanishes), which implies the estimate
on $dd^{c}\varphi_{t}.$ Next, the rate of convergence in \ref{eq:rate of conv in theorem envelope from dyn}
is proved by tracing through the proof of Prop \ref{prop:existence of the limit}.
Indeed, first the upper bound on $\varphi_{t}^{(\beta)}$ follows
from the uniform upper uniform bound of $dd^{c}\varphi_{t}$ in formula
\ref{eq:bounds in theorem env from dyn}, giving an error term of
the order $1/\beta.$ As for the lower bound it is obtained by taking
$\epsilon=C/\beta$ in the proof of Prop \ref{prop:existence of the limit}
and using that, since $\chi>0,$ the constant $C(T)$ can be taken
to be independent of $T$ (more precisely, one first fixes $x\in X$
and take $v(t)$ such that $v(x,t)\geq\varphi_{t}^{(\infty)}-\delta$
and finally let $\delta\rightarrow0).$
\end{proof}
In particular, by \ref{eq:rate of conv in theorem envelope from dyn}
\[
\sup_{X}\left|\varphi_{t}^{(\beta)}-P_{\omega}(f)\right|\leq C\left(\frac{\log\beta}{\beta}+e^{-t}\right)
\]
and hence the envelope $P_{\omega}(f)$ can be constructed from the
joint large $\beta$ and large $t-$limit of the Monge-Ampère flow
\ref{eq:ma flow for envelope}: 
\[
P_{\omega}(f):=\lim_{t\rightarrow\infty}\varphi_{t}^{(\beta_{t})}
\]
in the $C^{0}(X)-$norm for any family of $t-$dependent $\beta_{t}$
such that $\beta_{t}\rightarrow\infty$ as $t\rightarrow\infty.$
Interpreting $\beta_{t}$ as the ``inverse temperature'' this construction
is thus analogous to the method of simulated annealing algorithms
used in numerics to find nearly optimal global minima of a given energy
type function by cooling down a thermodynamical system (and decreasing
the corresponding free energy). The analogy can be made more precise
using the gradient flow picture in Section \ref{sec:The-gradient-flow}
where the energy functional in question is the pluricomplex energy
introduced in \cite{bbgz}. It would be interesting to see is numerically
useful in concrete situations, for example by adapting the numerical
implementations for the Kähler-Ricci flow on a toric manifold introduced
in \cite{d-h-h-k} (concerning a finite $\beta)$.

It may be illuminating to compare the dynamic construction of the
envelope $P_{\omega}(f)$ above with the dynamic PDE construction
of the\emph{ convex} envelope of a given smooth function $f$ on $\R^{n}$
introduced in \cite{ve}: 

\[
\frac{\partial\psi(t)}{\partial t}=\sqrt{1+|\partial_{x}\psi(t)|^{2}}\min\{0,\lambda_{1}(\partial_{x}^{2}\psi(t))\}\,\,\,\,\,\psi(0)=f,
\]
i.e. the graph of the solution $\psi_{t}$ evolves in the normal direction
at each point, with the speed $\min\{0,\lambda_{1}(\partial_{x}^{2}\psi(t))\}$
(expressed in terms of the first eigenvalue of the real Hessian $\partial_{x}^{2}\psi(t)));$
here $\psi(t)$ is a solution in the viscosity sense. A variant of
the latter construction, obtained by removing the first factor in
the right-hand side of the evolution equation above, was studied in
\cite{g-c} using stochastic calculus, where exponential convergence
was established with a uniform control bound on $\partial_{x}^{2}\psi(t)),$
which is thus analogous to the result in Theorem \ref{thm:envelope from dynamics}
above. Our approach can also be applied to convex envelopes by imposing
invariance in the imaginary directions (as in Section \ref{sub:Relation-to-the-krf}).
But the main difference in our setting is that we start with an arbitrary
convex function $\psi(0)$ and the dependence on $f$ instead appears
in the evolution equation itself. Moreover, the large parameter $\beta$
appears as a regularization parameter ensuring that the solution remains
smooth for positive times.

\subsection{\label{sub:The trivial case}The case when the class \textmd{$\frac{\text{1}}{\beta}c_{1}(K_{X})+[\theta_{\beta}]$}
is trivial}

Next we specialize to the case when $\frac{\text{1}}{\beta}c_{1}(K_{X})+[\theta_{\beta}]$
is trivial, which is the one relevant for the applications to Hele-Shaw
type flows and Hamilton-Jacobi equation (in the latter case $K_{X}$
is even trivial). Equivalently, this means that the non-normalized
KRF preserves the initial cohomology class. In particular, letting
$\beta\rightarrow\infty$ reveals that $[\theta]$ is trivial and
hence we can write 
\[
\theta=dd^{c}f,\,\,\,\,\inf_{X}f=0
\]
for a unique function $f$ and then take 
\[
\theta_{\beta}:=dd^{c}f_{\beta}+\frac{1}{\beta}\mbox{Ric }\omega,\,\,\,\,f_{\beta}:=f.
\]
 for a fixed Kähler form $\omega,$ i.e. by imposing the equation
\ref{eq:defining eq for f beta in terms of theta}. 

In this setting, the normalized flow always tends to zero as $t\rightarrow\infty$
(as the volume of the class does). But, by the seminal result in \cite{ca},
the non-normalized KRF flow converges to a Kähler form $\omega_{\beta}:$
\begin{prop}
\label{prop:cao}For a fixed $\beta>0$ the non-normalized Kähler-Ricci
flow $\omega^{(\beta)}(t)$ emanating from any given form $\omega_{0}$
converges (in the $C^{\infty}-$topology), as $t\rightarrow\infty,$
to the unique solution $\omega_{\beta}\in[\omega_{0}]$ of the Calabi-Yau
equation 
\begin{equation}
\frac{1}{V_{0}}\omega_{\beta}^{n}=\frac{e^{-\beta f}\omega^{n}}{\int_{X}e^{-\beta f}\omega^{n}},\label{eq:cy eq}
\end{equation}
where $V_{0}$ is the volume of $\omega_{0}.$ More precisely, under
the normalizations above the convergence holds on the level of Kähler
potentials. \end{prop}
\begin{rem}
By definition the volume form of the limiting Kähler metric is the
Boltzmann-Gibbs measure associated to the Hamiltonian function $f,$
at inverse temperature $\beta,$ which gives a hint of the statistical
mechanical interpretation of the large $\beta-$limit (see Section
\ref{sec:The-gradient-flow} for further hints). 
\end{rem}
It should be stressed that by the estimates in \cite{ca} one has
in this setting that 
\begin{equation}
\omega^{(\beta)}(t)\leq C_{\beta}\label{eq:bound in cy krf}
\end{equation}
\label{eq:upper bound on omega cao} independently of $t,$ which
seemingly improves on the bounds in Theorem \ref{thm:main general}
and Theorem \ref{thm:sharp bounds} for large $t$ (the proof uses
a different application of the Laplacian estimate, along the lines
of Yau's original argument, which needs a two-sided bound on the potential).
But the point of the estimates in Theorem \ref{thm:sharp bounds},
where one gets a linear growth in $t$ is to get a multiplicative
constant that is independent of $\beta$ (at least when $t\rightarrow\infty).$
In fact, for a generic $f,$ it is impossible to get a constant $C_{\beta}$
in formula \ref{eq:bound in cy krf} which is independent of $\beta.$
Indeed, unless $f$ vanishes identically the Gibbs measure in the
right-hand side of the Calabi-Yau equation \ref{eq:cy eq}  blows
up as $\beta\rightarrow\infty,$ concentrating on the subset of $X$
where $f$ attains its absolute minimum $(=0$ with our normalizations).
Hence, for a generic $f$ any limit point of $\omega^{(\beta)}(\infty)$
is a sum of Dirac measures. Accordingly, the convergence in the previous
proposition, motivates (by formally interchanging the large $t$ and
large $\beta-$limits) the following 
\begin{prop}
\label{prop:conv of p theta volume form in cy case} Let $f$ be a
smooth function on $X$ and $\omega_{0}$ a Kähler form on $X.$ Then
any limit point of the family $(P(\omega_{0}+tdd^{c}f))^{n}$ (in
the weak topology) is supported in the closed set $F$ where $f$
attains its absolute minimum. In particular, 
\begin{itemize}
\item if $f$ admits a unique absolute minimum $x_{0}$ then 
\[
\lim_{t\rightarrow\infty}((P(\omega_{0}+tdd^{c}f))^{n}=V_{0}\delta_{x_{0}}
\]
weakly. Hence, the corresponding non-normalized Kähler-Ricci flows
$\omega^{(\beta)}(t$) emanating from $\omega_{0}$ satisfy 
\[
\lim_{t\to+\infty}\lim_{\beta\text{\ensuremath{\to}+\ensuremath{\infty}}}\omega^{(\beta)}(t)^{n}=V_{0}\delta_{x_{0}}
\]
 in the weak topology.
\item In general, under the normalization $\inf_{X}f=0,$ 
\[
\lim_{t\to+\infty}P_{\omega}(\varphi_{0}+tf)=P_{(\omega,F)}(\varphi_{0})
\]
(increasing pointwise) for any initial continuous $\omega$-psh function
$\varphi_{0}$. In particular, if $F$ is not pluripolar, then \textup{
\[
\lim_{t\rightarrow\infty}\lim_{\beta\rightarrow\infty}\omega^{(\beta)}(t)=\omega_{\infty}
\]
}in the weak topology, where $\omega_{\infty}$ is the positive current
defined by $\omega_{0}+dd^{c}P_{(\omega_{0},F)}(0).$
\end{itemize}
\end{prop}
\begin{proof}
First observe that, under the normalization $f(x_{0})=0$ and $f>0$
we have $P_{\chi}(f)=0,$ as in this case $\chi=0$ and all psh functions
on a compact manifold $X$ are constants. Next, by Lemma \ref{lem:conv of envelopes}
$P_{\hat{\omega}_{t}}((1-e^{-t})f):=P_{t}(f_{t})\rightarrow P_{\chi}f$
in the $L^{1}-$topology. Since $P_{t}(f_{t})$ is in $PSH(X,C\omega)$
for $C$ sufficiently large it follows from basic properties of psh
functions that $\sup_{X}P_{t}(f_{t})\rightarrow\sup_{X}P_{\chi}(f)=0.$
Fixing $\epsilon>0$ this means that for $t\geq t_{\epsilon}$, $\sup_{X}P_{t}(f_{t})\leq\epsilon/2$
and hence the non-coincidence sets $\Omega_{t}$ satisfy $\{f>\epsilon\}\subset\Omega_{t}$
for $t\geq t_{\epsilon}.$ In particular, $(\hat{\omega}_{t}+dd^{c}P_{\hat{\omega}_{t}}((1-e^{-t})f))$
and hence its non-normalized version $(\omega_{0}+dd^{c}P_{\omega_{0}}(tf))^{n}$
is supported in $\{f\leq\epsilon\}$ for $t\geq t_{\epsilon},$ which
concludes the proof of the first statement. The first point then follows
immediately. 

To prove the second point we may assume that $\inf_{X}f=0$, hence
the family $P_{\omega}(\varphi_{0}+tf)$ is increasing in $t$. By
assumption $P_{\omega}(\varphi_{0}+tf)\leq\varphi_{0}+tf=\varphi_{0}$
on $F$ and hence $P_{\omega}(\varphi_{0}+tf)\leq P_{(\omega,F)}(\varphi_{0}).$
To prove the reversed inequality we fix $\varepsilon>0$ and $u\in PSH(X,\omega)$
such that $u\leq\varphi_{0}$ on $F$. Since the sets $\left\{ f\leq c\right\} $
decrease to the compact set $F$ as $c\downarrow0$ and $\varphi_{0}$
is continuous, there exists $c>0$ small enough such that $\left\{ f\leq c\right\} \subset\left\{ v-\varepsilon<\varphi_{0}\right\} $.
Now, for $t>c^{-1}\sup_{X}(v-\varphi_{0})$ we have $v-\varepsilon\leq\varphi_{0}+tf$,
giving that the limit $\varphi_{\infty}$ of the increasing family
$P_{\omega}(\varphi_{0}+tf)$ is greater than $v-\varepsilon$. As
$v$ and $\varepsilon$ were chosen arbitrarily the conclusion follows.
\footnote{The same result holds even when $F$ is non-pluripolar and $\varphi_{0}$
is unbounded (using the domination principle in finite energy classes
due to Dinew \cite{BL12}), but we are not going further into this
here. }.
\end{proof}

\subsection{\label{sub:Comparison-with-convergence}Comparison with convergence
properties for a finite $\beta$ and canonical metrics}

\subsubsection{The big case}

Let us start by considering the case when $\theta=0.$ Then, up to
a scaling, we may as well also assume that $\beta=1.$ When $K_{X}$
is nef and big, which equivalently means that $K_{X}$ is semi-positive
(by the base point freeness theorem) and with non-zero volume, $K_{X}^{n}>0,$
it is well-known that the normalized Kähler-Ricci flow emanating from
any given Kähler metric $\omega_{0}$ on $X$ converges, weakly in
the sense of currents, to the unique (possible singular) Kähler-Einstein
metric (or rather current) $\omega_{KE}$ on $X$ \cite{ts,t-z}.
This fact implies the following
\begin{prop}
\label{prop:no linear bd} Assume that $K_{X}$ is nef and big, but
not ample. Then it is not possible to have an upper bound of the form
$\omega^{(\beta)}(t)\leq C_{\beta}t$ along the non-normalized Kähler-Ricci
flow, for $t$ large.\end{prop}
\begin{proof}
Fixing a semi-positive form $\chi$ in $c_{1}(K_{X})$ and representing
$\omega_{KE}=\chi+dd^{c}\varphi_{KE}$ the potential $\varphi_{KE}$
may be characterized as the unique continuous solution in $PSH(X,\chi)$
to the equation 
\[
(\chi+dd^{c}\varphi)^{n}=e^{\varphi}dV_{\chi}
\]
(in the sense of pluripotential theory) where $dV_{\chi}$ is the
normalized volume form determined by $\chi$ (i.e. $\mbox{Ric \ensuremath{dV_{\chi}=\chi)}.}$
In particular, if $K_{X}$ is not positive (i.e. not ample) then $\omega_{KE}$
is not a bounded current. Indeed, assuming to get a contradiction
that $\omega_{KE}\leq C\omega_{0}$ the previous equation gives that
$\omega_{KE}^{n}\geq\delta\omega_{0}^{n}$ for some positive constant
$\delta$. But this means that, up to enlarging the constant $C$
we get $\omega_{0}/C\leq\omega_{KE}\leq C\omega_{0}$ which forces
$K_{X}$ to be ample (for example, by the Nakai-Moishezon criterion
or by a direct regularization argument).
\end{proof}
More generally, essentially the same arguments apply to any smooth
twisting form $\theta$ and parameter $\beta$ as long as $K_{X}/\beta+[\theta]$
is nef and big.

\subsubsection{The non-big case }

Again we start with the case when $\theta=0$ with $K_{X}$ nef, but
now not big. Assuming that the abundance conjecture holds, i.e. that
$K_{X}$ is semi-ample it was shown in \cite{s-t-1} that the normalized
Kähler-Ricci flow, emanating from any given Kähler metric $\omega_{0},$
on $X$ converges, weakly in the sense of currents, to a canonical
current $\omega_{X}$ on $X$ defined as follows: by the semi-ampleness
assumption there exists a holomorphic map $F$ from $X$ to a variety
$Y$ such that $K_{X}=F^{*}A$ where $A$ is an ample line bundle
on $Y.$ In case $Y$ is zero-dimensional the limit $\omega_{X}$
vanishes identically (as in Section \ref{sub:The trivial case}).
Otherwise, denoting by $\kappa$ the dimension of $Y$ (which equals
the Kodaira dimension of $X)$, picking a Kähler form $\omega_{A}$
in $c_{1}(A)$ and taking $\chi:=F^{*}\omega_{A},$ the limiting current
$\omega_{X}$ obtained in \cite{s-t-1} can be realized as $F^{*}(\omega_{A}+dd^{c}\psi)$
where $\psi$ is the unique continuous solution in $PSH(Y,\omega_{A})$
of the equation 
\[
(\omega_{A}+dd^{c}\psi)^{n}=e^{\psi}F_{*}(dV_{\chi}).
\]
 Next, we make some heuristic remarks about the connection to the
double limit in formula \ref{eq:double scaling limit}. We assume
that $K_{X}$ is semi-ample and fix a smooth form $\theta$ in $c_{1}(K_{X}),$
a Kähler metric $\omega$ on $X$ and define $\theta_{\beta}$ and
$f$ by 
\[
\theta_{\beta}:=\theta-\frac{1}{\beta}\mbox{Ric \ensuremath{\omega}},\,\,\,\theta=dd^{c}f+\chi.
\]
In particular, $c_{1}(K_{X})/\beta+[\theta_{\beta}]=c_{1}(K_{X})$
for all $\beta.$ In the light of the result in \cite{s-t-1} one
would expect that the corresponding twisted normalized KRF $\omega^{(\beta)}(t)$
converges, as $t\rightarrow\infty,$ to the current $F^{*}(\omega_{A}+dd^{c}\psi_{\beta})$,
where $\psi_{\beta}$ is the unique continuous solution in $PSH(X,\chi)$
of the equation 
\[
(\omega_{A}+dd^{c}\psi_{\beta})^{\kappa}=e^{\beta\psi_{\beta}}F_{*}(e^{-\beta f}\omega^{n})
\]
We will make the hypothesis that this is the case. It can be shown
that as $\beta\rightarrow\infty$ there exist a (mildly singular)
volume form $\mu_{Y}$ on $Y$  such that 
\[
F_{*}(e^{-\beta f}dV_{\chi})=e^{-\beta(\bar{f}+o(1))}\mu_{Y},
\]
where $\bar{f}(y):=\inf_{F^{-1}(\{y\})}f$ (using that the push forward
$F_{*}$ amounts to integration along the fibers of $F$ which thus
picks out the infimum of $f$ over the fibers as $\beta\rightarrow\infty$;
the error term $o(1)$ is uniform away from the branching locus of
the map $F).$ But then a variant of Theorem \ref{thm:envelope from dynamics}
(see \cite{berm5}) shows that $\psi_{\beta}\rightarrow\psi_{\infty}:=P_{\omega_{A}}(\bar{f})$
and hence, under the hypothetical convergence above,
\[
\lim_{t\rightarrow\infty}\lim_{\beta\rightarrow\infty}\omega^{(\beta)}(t)=\chi+dd^{c}P_{F^{*}\omega_{A}}(F^{*}\bar{f}).
\]
Finally, since $K_{X}=F^{*}A$ we have $PSH(X,\chi)=F^{*}PSH(Y,\omega_{A}),$
forcing $P_{F^{*}\omega_{A}}(F^{*}\bar{f})=P_{\chi}(f),$ i.e. the
rhs above is equal to the current obtained by interchanging the limits
in the lhs (as in Lemma \ref{lem:conv of envelopes}), i.e. the two
limits may be interchanged under the hypothesis above.

\section{\label{sec:Relations-to-viscosity}Applications to Hamilton-Jacobi
equations and shocks }

\subsection{Background}

Let $H$ be a smooth function on $\R^{n}.$ The corresponding \emph{Hamilton-Jacobi
equation} (with \emph{Hamiltonian} $H)$ is the following evolution
equation

\begin{equation}
\frac{\partial\psi_{t}(y)}{\partial t}+H(\nabla\psi_{t}(y))=0,\,\,\,\,\psi_{|t=0}=\psi_{0}\label{eq:hj eq}
\end{equation}
for a function $\psi(x,t)$ on $\R^{n}\times[0,\infty[.$ It is a
classical fact that, even if the initial function $\psi_{0}$ is smooth
a solution $\psi_{t}$ typically develops shock singularities at a
finite time $T_{*},$ i.e. it ceases to be differentiable in the space-variable
(due to the crossing of characteristics). In order to get a solution
defined for any positive time the notion of \emph{viscosity solution
}was introduced in \cite{c-l,c-e-l}. The momentary \emph{shock locus}
$S_{t}$ of such a solution $\psi_{t}$ is defined by
\[
S_{t}:=\{x:\,\,\psi_{t}\,\mbox{ is\  not\,\ differentiable\,\ at\,\ensuremath{x}\}}.
\]
When $H$ is convex the classical \emph{Hopf-Lax formula} provides
an explicit envelope expression for a viscosity solution of the Cauchy
problem for the HJ-equation \ref{eq:hj eq} with any given smooth
initial data $\psi_{0}$ (which, for example, appears naturally in
optimal control problems): 
\[
\psi_{t}(y)=\inf_{x\in\R^{n}}\psi_{0}(x)+tH^{*}\left(\frac{x-y}{t}\right),
\]
expressed in terms of the Legendre transform: 
\[
g^{*}(y):=\sup_{x\in\R^{n}}x\cdot y-g(x).
\]
On the other hand, in the case when $H$ is non-convex, but the initial
data $\psi_{0}$ is assumed convex, the \emph{second Hopf formula}
\cite{b-e,l-r} provides a viscosity solution which may be represented
as 
\begin{equation}
\psi_{t}=(\psi_{0}^{*}+tH)^{*}.\label{eq:second hopf f}
\end{equation}
This was shown in \cite{b-e} using the theory of differential games
and in \cite{c-l} by a more direct approach. In particular the viscosity
solution $\psi_{t}$ above remains convex for all positive times and
as a consequence its shock locus $S_{t}$ is a codimension one hypersurface
with singularities (or empty, as is the case for small $t).$ 

We recall that the viscosity terminology can be traced back to the
fact that viscosity solutions may often be realized as limits of smooth
solutions $\psi^{(\beta)}$ of the following perturbed (viscous) HJ-equations
(where the constant $\beta^{-1}$ plays the role of the viscosity
constant in fluid and gas dynamics): 
\begin{equation}
\frac{\partial\psi_{t}(y)}{\partial t}+H(\nabla\psi_{t}(y))=\frac{1}{\beta}\Delta\psi_{t}(u)\label{eq:pert hj equ with linear v}
\end{equation}
 as $\beta\rightarrow\infty.$ For example, the following result holds:
\begin{thm}
\label{thm:.(vanishing-viscosity-limit).}\cite[Theorem 3.1]{c-l,c-e-l}.
(Vanishing viscosity limit). Assume that $\psi_{t}^{(\beta)}$ are
smooth solutions to the previous equation and that a subsequence converges
uniformly to $\psi_{t}.$ Then $\psi_{t}$ is a viscosity solution
to the HJ-equation (\ref{eq:hj eq}). 
\end{thm}
In particular, under suitable growth assumptions, ensuring that the
viscosity solution $\psi_{t}$ is uniquely determined, the whole family
converges to $\psi_{t}.$ Note however that, in general, $\Delta\psi_{t}^{(\beta)}$
will not be uniformly bounded (even locally), as this would entail
that the limit $\psi_{t}$ is differentiable on $\R^{n}.$

Next, we make the observation that in the case when the initial data
$\psi_{0}$ above is taken to be $\frac{|y|^{2}}{2}$ the second Hopf
formula is equivalent to the Hopf-Lax formula for the convex Hamiltonian
$\frac{|x|^{2}}{2}:$
\begin{lem}
\label{lem:hopf-lax and 2 hopf} Let $\Phi_{0}$ be a given function
on $\R^{n}$ and denote by $\Phi_{t}$ the Hopf-Lax viscosity solution
to the HJ-equation with convex Hamiltonian $|x|^{2}/2$ and initial
data $\Phi_{0}.$ Then 
\[
\psi_{t}(y):=\left(-\Phi_{t}(y)t+\frac{|y|^{2}}{2}\right)
\]
gives the viscosity solution to the HJ-equation with non-convex Hamiltonian
$H:=\Phi_{0}$ and initial data $\psi_{0}(y):=\frac{|y|^{2}}{2}$
provided by the second Hopf formula (and conversely). In particular,
the shock loci of $\Phi_{t}$ and $\psi_{t}$ coincide. \end{lem}
\begin{proof}
This follows immediately from comparing the Hopf-Lax formula and the
second Hopf formula. 
\end{proof}
The previous lemma is consistent (as it must) with the fact that when
$\psi_{0}(y)=\frac{|y|^{2}}{2}$ the Hamiltonian $H$ can, by the
definition of the HJ-equation, be recovered as minus the derivative
at $t=0$ of the corresponding viscosity solution $\psi_{t}.$

\subsubsection{The adhesion model in cosmology }

The convex case where $H(x)=|x|^{2}/2$ is ubiquitous in mathematical
physics and appears, in particular, in the adhesion model for the
formation of the large-scale structure in the early universe (known
as the ``cosmic web'') where $\Phi_{0}$ is proportional to the
gravitational potential of the initial fluctuations of the density
field and the shock region $S_{t}$ corresponds to emerging regions
of localized mass concentration (the adhesion model is an extension
of the seminal Zel'dovich approximation beyond $t\geq T_{*}$) \cite{g-m-s,v-d-f-n,hsw}.
The corresponding singularities of $S_{t}$ and their metamorphosis
as $t$ evolves have been classified in dimensions $n\leq3$, for
generic initial data, using the catostropy theory of Lagrangian singularities
initiated by Arnold \cite{a-s-z,g-m-s,bo,hsw,k-p-s-m}. In this setting
the Legendre transform $\phi_{t}:=\psi_{t}^{*}$ of the corresponding
function $\psi_{t}$ appearing in the previous lemma is given by 
\[
\phi_{t}(x)=x+t\Phi_{0}(x)
\]
and the corresponding map
\begin{equation}
x\mapsto\nabla_{x}\phi_{t}(x)\label{eq:Zeldov map}
\end{equation}
 describes, in the Zeldovich approximation, the displacement of a
particle with initial coordinate $x$ to the position $y$ a time
$t$ (in the physics literature the initial coordinate space $x$
is called the Lagrangian space and the position space $y$ at time
$t$ is called the Euler space; accordingly $\phi_{t}$ is often called
the Lagrangian potential). The map above is injective precisely for
$t<T_{*}.$ In the next section we will show that the adhesion model
can be realized as the zero-temperature limit of the twisted Kähler-Ricci
flow (using Lemma \ref{lem:hopf-lax and 2 hopf}). 
\begin{rem}
\label{rem:burger} When $H(x)=|x|^{2}/2$ the vector field $v_{t}(y):=\nabla u_{t}(y)$
determined by a solution $u_{t}$ of the corresponding HJ-equation
satisfies Burger's equation:
\[
\frac{\partial v_{t}(y)}{\partial t}+\frac{1}{2}\nabla\left|v_{t}(y)\right|^{2}=0,
\]
 which is the prototype of a hyperbolic conservation law \cite{se}
and non-linear wave phenomena \cite{g-m-s}. 
\end{rem}

\subsection{\label{sub:Relation-to-the-krf} Relation to the Kähler-Ricci flow
and Theorem \ref{thm:main general}}

The relation between the Hamilton-Jacobi equation and the Kähler-Ricci
flow, which does not seem to have been noted before, arises when the
linear viscosity term in the perturbed HJ-equation \ref{eq:pert hj equ with linear v}
is replaced by the following non-linear one:

\begin{equation}
\frac{\partial\psi_{t}(y)}{\partial t}+H(\nabla\psi_{t}(y))=\frac{1}{\beta}\log(\partial^{2}\psi_{t}(y))\label{eq:hj equation with nonlinear visc}
\end{equation}
for $\psi_{0}$ strictly convex (for example, $\psi_{0}(y)=|y|^{2}/2,$
as in the adhesion model above). One virtue of the latter evolution
equation is that, as will be shown below, the smooth solution $\psi_{t}^{(\beta)}$
remains convex (and even strictly so) for positive times.

We will consider the case when the Hamiltonian $H$ is periodic, i.e.
invariant under the action of a lattice $\Lambda$ on $\R^{n}$ by
translations\footnote{It seems likely that the general case could be studied by extending
our results to (appropriate) non-compact manifolds $X$ or by approximation.}. Without loss of generality we may and will assume that a fundamental
domain for $\Lambda$ has unit volume. Since there are no non-constant
periodic convex functions on $\R^{n}$ the natural condition on the
initial function $\psi_{0}$ is that it is in the class of all convex
functions $u$ which are \emph{quasi-periodic} in the sense that $\psi(y)-|y|^{2}/2$
is $\Lambda-$periodic on $\R^{n}.$ We denote by $\mathcal{C}_{\Lambda}$
the the space of all quasi-periodic convex functions on $\R^{n}.$
The point is that for any $\psi\in\mathcal{C}_{\Lambda}$ the Hessian
$\partial^{2}\psi$ is periodic and the gradient map $\partial\psi$
is $\Lambda-$equivariant and hence all terms appearing in the equation
\ref{eq:hj equation with nonlinear visc} are $\Lambda-$periodic. 
\begin{lem}
\label{lem:the space of convex} Equip the space $\mathcal{C}_{\Lambda}$
with the sup-norm. Then the Legendre transform $\phi\mapsto\psi:=\phi^{*}$
induces an isometry on $\mathcal{C}_{\Lambda}$ and for any quasi-periodic
function $f$ 
\[
\sup_{\phi\leq f}\{\phi\}=f^{**},
\]
 where the sup, that we shall denote by $Pf,$ can be taken either
over all convex functions $\phi$ or over all quasi-periodic convex
functions. Moreover, the subspace of all $\phi$ in $\mathcal{C}_{\Lambda}$
such that $\sup_{x\in\R}(\phi(x)-|x|^{2}/2)=0$ is compact. \end{lem}
\begin{proof}
The isometry property follows directly from the relation $(\phi+c)^{*}=\phi^{*}-c.$
Next, if $P'f$ denotes the sup over all convex $\phi$ below $f$
then, by the extremal property, the function $P'f-\frac{|x|^{2}}{2}$
has to be $\Lambda-$periodic, as $f-\frac{|x|^{2}}{2}$ is, i.e.
$P'f$ is quasi-periodic, as desired. Finally, it is well-known that
if $f$ is convex then $P'f=f^{**}.$ 

The compactness is a consequence of the Arzela-Ascoli theorem and
the fact that if $\phi$ is in $\mathcal{C}_{\Lambda}$ then the periodic
function $(\phi(x)-|x|^{2}/2)$ is $L-$Lipschitz for a constant $L$
only depending on the diameter of a fundamental domain of $\Lambda;$
see \cite[Lemma 3.14]{hu} where further properties of the space $\mathcal{C}_{\Lambda}$
are also established. 
\end{proof}
In this setting Theorem \ref{thm:main general} admits the following
dual formulation:
\begin{thm}
\label{thm:conv towards hopf on torus} Consider the perturbed HJ-equation
\ref{eq:hj equation with nonlinear visc} with $\Lambda-$periodic
smooth Hamiltonian $H$ and strictly convex and quasi-periodic initial
data $\psi_{0}.$ Denote by $\psi^{(\beta)}$ the unique solution
of the corresponding Cauchy problem such that $\psi_{t}^{(\beta)}$
is quasi-periodic and strictly convex. Then $\psi_{t}^{(\beta)}$
converges, as $t\rightarrow\infty,$ uniformly in space, to $\psi_{t}$
given by the second Hopf formula \ref{eq:second hopf f}, which is
the\textup{ unique viscosity solution of the HJ-equation \ref{eq:hj eq}
with initial data $\psi_{0}.$} Moreover, $\psi_{t}^{(\beta)}$ is
strictly convex for any $t>0,$ uniformly in $\beta:$ 
\[
\left\Vert \partial^{2}\psi_{t}^{(\beta)}\right\Vert \geq\frac{1}{t+1}\min\{\left\Vert \partial^{2}\psi_{0}\right\Vert ,\left\Vert \partial^{2}H\right\Vert ),
\]
 in terms of the trace norm defined wrt the Euclidean metric (i.e
the sup on $\R^{n}$ of the point-wise $L^{1}-$norm). 
\end{thm}
To make the connection to the complex geometric setting we let $X$
be the abelian variety $X:=\C^{n}/(\Lambda+i\Z^{n})$ and consider
the the following holomorphic $T-$action on $X:$
\[
([x+iy],[a])\mapsto[x+iy+a],
\]
 where $T$ denotes the real $n-$torus $T:=\R^{n}/\Z^{n}$ and $\pi(z):=[z]$
denotes the corresponding quotient map. Let $\omega$ be the standard
flat Kähler metric on $X$ induced from the Euclidean metric $\omega_{0}$
on $\C^{n}$ normalized so that $\omega_{0}=dd^{c}|x|^{2}/2$ and
fix a closed $T-$invariant $(1,1)-$ form $\theta$ which is exact,
i.e. 
\[
\theta=dd^{c}f
\]
 for a $T-$invariant function $f$ on $X$ (uniquely determined up
to an additive constant). Now we can identify $T-$invariant elements
in $PSH(X,\omega)$ with convex functions $\phi(x)$ on $\R^{n}$
in the space $\mathcal{C}_{\Lambda}$ (by setting $\phi:=|x|^{2}/2+\pi^{*}\varphi$
and using that $dd^{c}(|x|^{2}/2+\pi^{*}\varphi)=\omega_{0}+dd^{c}\pi^{*}\varphi\geq0).$
Accordingly, the non-normalized KRF in the class $[\omega]$ with
twisting form $\theta$ thus gets identified with the following parabolic
equation on $\R^{n}:$ 

\begin{equation}
\frac{\partial\phi_{t}^{(\beta)}(x)}{\partial t}=\frac{1}{\beta}\log(\partial^{2}\phi_{t}^{(\beta)}(x))+H(x),\label{eq:kfr in convex setting}
\end{equation}
where $H$ is the $\Lambda-$periodic function on $\R^{n}$ corresponding
to $f$ and $\phi_{t}^{(\beta)}\in\mathcal{C_{\Lambda}}.$ More precisely,
$\phi_{t}^{(\beta)}$ is smooth and strictly convex. The key observation
now is that setting 
\[
\psi_{t}^{(\beta)}(y):=\phi_{t}^{(\beta)*}(y)
\]
gives a solution in $\mathcal{C}_{\Lambda}$ to the perturbed HJ-equation
\ref{eq:hj equation with nonlinear visc}. Indeed, this follows from
the following well-known properties of the (involutive) Legendre transform
between smooth and strictly convex functions (say with quadratic growth
at infinity):
\begin{equation}
\partial^{2}\phi(x)=(\partial^{2}\psi(y))^{-1},\,\,\,\frac{\partial(\phi+tv)(x)}{\partial t}_{|t=0}=-v(\partial_{y}\psi(y)),\,\,\,\,y:=\partial_{x}\phi(x)\label{eq:legendre relations}
\end{equation}
(see for example the appendix in \cite{b-b1} for a proof of the latter
formula). Now, by Theorem \ref{thm:main general} and the previous
lemma 
\[
\lim_{\beta\rightarrow\infty}\phi_{t}^{(\beta)}=P_{\Lambda}(\phi_{0}+tH)=(\phi_{0}+tH)^{**}
\]
in $\mathcal{C}_{\Lambda}.$ Since the Legendre transform is an isometry
on $\mathcal{C}_{\Lambda}$ and in particular continuous this equivalently
means that $\lim_{\beta\rightarrow\infty}\psi_{t}^{(\beta)}=(\phi_{0}+tH)^{*},$
which coincides with the viscosity solution of the HJ-equation provided
by the second Hopf formula. Finally, the proof of the previous theorem
is concluded by noting that the uniqueness of viscosity solutions
in $\mathcal{C}_{\Lambda}$ follows from the standard uniqueness argument
\cite{c-l,c-e-l}, using that for any two functions in $\mathcal{C_{\Lambda}}$
the difference $u-v$ is continuos and attains its maximum and minimum
(since it is periodic). 

In fact, in this way Theorem \ref{thm:main general} could be used
to give an alternative proof of the fact that the second Hopf formula
defines a viscosity solution to the HJ-equation \ref{eq:hj eq}, by
adapting the proof of Theorem \ref{thm:.(vanishing-viscosity-limit).}
to the present non-linear setting. But we will not go further into
this here.
\begin{rem}
\label{rem:r-z} Convex envelopes of the form $\psi_{t}:=(\psi_{0}^{*}+tH)^{*}(=\phi_{t}^{*})$
and the corresponding sets $X(t)$ also appear in a different Kähler-geometric
setting in \cite{r-wn1,rzIII}, where it is shown that $\psi_{t}$
defines a torus invariant (weak) Kähler geodesic precisely on $[0,T_{*}[$
(what we call $T_{*}$ is called the ``convex life span'' in \cite{r-wn1,rzIII}).
By definition, such a Kähler geodesic $\phi_{t}$ is characterized
by the homogeneous Monge-Ampère equation $MA(\phi)=0$ on the product
$X\times]0,T[.$ The relation to ($C^{1}-$smooth) solutions of Hamilton-Jacobi
equations was also pointed out in Section 6 in \cite{rzIII}. In the
light of the results in \cite{r-wn1,rzIII} it seems notable that
in our setting $\phi_{t}$ has a natural complex geometric interpretation
also for $t>T_{*}$ (namely, as a limiting Kähler-Ricci flow).
\end{rem}

\subsubsection{\label{rem:conv duality} Remarks on convex duality in the present
setting }

By a well-known duality principle in convex analysis differentiability
of a convex functions $\psi$ corresponds, loosely speaking, to strict
convexity of its Legendre transform $\phi:=\psi^{*}.$ To make this
precise we will assume that both $\phi$ and $\psi$ are defined on
all of $\R^{n}$ and have super-linear growth (which is the case when
any, and hence both, of the functions are in $\mathcal{C}_{\Lambda}).$
This ensures that the sub gradient maps $\partial\phi$ and $\partial\psi$
are both surjective. We recall that a convex function $\phi$ is differentiable
at $x$ iff the subgradient $(\partial\phi)(x)$ is single valued
and then we will write $(\partial\phi)(x)=(\nabla\phi)(x).$ The starting
point for the duality in question is the following fact (which follows
directly from the definitions): 
\[
x\in\partial\psi(y)\iff y\in\partial\phi(x)\iff x\cdot y=\phi(x)+\psi(y)
\]
In our setting $\phi:=\phi_{t}$ (for a fixed time $t)$ is $C^{1,1}-$smooth,
i.e. $\partial\phi(=\nabla\phi)$ defines a surjective Lipchitz map
$\R^{n}\rightarrow\R^{n}.$ As a consequence, a point $y$ is in the
shock locus $S_{t}$ of $\psi_{t}$ iff $y\in\partial\phi_{t}(U),$
for an open set $U$ where the Lipchitz map $\partial\phi_{t}$ is
not injective (which can be interpreted as a local strict convexity
of $\phi_{t}$). Let now $X_{t}$ be the support of the Monge-Ampère
measure $\det(\partial^{2}\phi_{t})dx$ and denote by $\Omega_{t}$
its complement. For simplicity we assume that the locus where $\phi_{t}$
is in $C_{loc}^{2}$ is dense in $\R^{n}$ (which presumably holds
for a generic $H$ using the arguments in \cite{a-s-z,bo}). In that
case the continuous map $\partial\phi_{t}$ maps the interior of $X_{t}$
injectively to $\mathbb{R}^{n}-S_{\psi_{t}}$ and $\overline{\Omega}_{t}$
non-injectively to $S_{\psi_{t}}$ (since a $C^{2}-$convex function
$u$ has an invertible gradient iff $\det(\partial^{2}u)>0$). Conversely,
$\nabla\psi_{t}$ maps $\R-S_{\psi_{t}}$ to $X_{t}.$ See for example
\cite[Fig 5]{v-d-f-n} for an illustration of this duality. 

It may also be illuminating to consider the case when $\psi$ is piece-wise
affine (which, as we will show in the next section, happens when $t=\infty).$
Then $(\nabla\phi)(\mathbb{R}^{n}-S_{\phi})$ is contained in the
$0-$dimensional stratum $S_{\psi}^{(0)}$ of $S_{\psi}$ (i.e. in
the vertex set). Indeed, if $y_{0}:=(\nabla\phi)(x_{0})$ is not in
$S_{\psi}^{(0)}$ then there is an open affine segment $L$ passing
through $y_{0}$ along which $\psi$ is affine. One then gets a contradiction
to the the differentiability of $\phi$ at $x_{0}$ by noting that
$L\subset\partial\phi(x_{0}).$ Indeed, since $x_{0}\in\partial\psi(y_{0})$
one gets $\psi(y)=\psi(y_{0})+x_{0}\cdot(y-y_{0})$ along $L.$ But
this means that $x_{0}\cdot y=\phi(x_{0})+\psi(y)$ and hence $y\in\partial\phi(x_{0}).$ 

In fact, this argument also shows that $\phi$ is piecewise affine
iff its Legendre transform $\psi$ is. Indeed, if $\psi$ is piecewise
affine then by the growth assumptions the sup defining $\phi$ is
always attained. Hence, for any $x\in\R^{n}-S_{\phi}$ we have that
$\phi=(\chi_{S_{\psi}^{(0)}}\psi)^{*}.$ Since the rhs is also a convex
function and the complement of $\R^{n}-S_{\phi}$ is a null set it
then follows that $\phi=(\chi_{S_{\psi}^{(0)}}\psi)^{*}$ everywhere,
showing that $\phi$ is also piece-wise affine, as desired.

\subsection{The large time limit and Delaunay/Voronoi tessellations }

Next, we specialize the large time convergence result in Prop \ref{prop:conv of p theta volume form in cy case}
to the present setting, showing, in particular, that the Hessian of
the limiting solution vanishes almost everywhere:
\begin{thm}
\label{thm:conv towards Vor} Denote by $F_{\Lambda}$ the closed
set in $\R^{n}$ where the $\Lambda-$periodic Hamiltonian $H$ attains
its minimum, normalized to be $0$ and assume that $F_{\Lambda}$
is discrete. Then, for any given initial data $\psi_{0}$ in the space
$\mathcal{C}_{\Lambda}$ the unique viscosity solution $\psi_{t}$
in $\mathcal{C}_{\Lambda}$ of the corresponding Hamilton-Jacobi equation
converges uniformly to the following convex piecewise affine function:
\begin{equation}
\psi_{\infty}(y):=\sup_{x\in F_{\Lambda}}x\cdot y-\psi_{0}^{*}(x).\label{eq:limit of hj}
\end{equation}
Equivalently, the large $\beta-$limit $\phi_{t}$ of the Kähler-Ricci
flow \ref{eq:kfr in convex setting} converges to the convex piecewise
affine function $\phi_{\infty}(x)$ whose graph is the convex hull
of the discrete graph of the function $\phi_{0}$ restricted to $F_{\Lambda}.$\end{thm}
\begin{proof}
By the second point in Prop \ref{prop:conv of p theta volume form in cy case}
\[
\phi_{\infty}(x):=\sup_{\phi\in\mathcal{C}_{\Lambda}}\left\{ \phi(x):\,\,\phi\leq\phi_{0}\mbox{\,\ on\,\ensuremath{F_{\Lambda}}}\right\} .
\]
Indeed, recall that the limit in Prop \ref{prop:conv of p theta volume form in cy case}
is the supremum over all $\omega$-psh functions lying below $\chi_{F_{\Lambda}}\phi_{0}$.
But as $F_{\Lambda}$ is non-pluripolar (which follows from the classical
fact in pluripotential theory that $\mathbb{R}^{n}$ is non-pluripolar
in $\mathbb{C}^{n}$), the function $\phi_{\infty}$ is convex bounded
in $\mathbb{R}^{n}$. This together with the maximality property yields
that $\phi_{\infty}$ is $T$-invariant and hence the corresponding
function in $\mathcal{C}_{\Lambda}$ equals the supremum taken over
$\mathcal{C}_{\Lambda}$ as above. Alternatively, the boundedness
can also be seen directly in the present setting using the compactness
property in Lemma \ref{lem:the space of convex}. Writing this as
$\phi_{\infty}=P(\chi_{F_{\Lambda}}\phi_{0}),$ where $\chi_{F_{\Lambda}}=0$
on $F_{\Lambda}$ and $+\infty$ on the complement of $F_{\Lambda}$
(compare Lemma \ref{lem:the space of convex}) reveals that the previous
sup coincides with the relaxed sup $\phi'$ obtained by simply requiring
that $\phi$ be convex (but not quasi-periodic), i.e. the graph of
$\phi_{\infty}(x)$ is the convex hull of the discrete graph of the
function $\phi_{0}$ restricted to $F_{\Lambda},$ as desired. By
Lemma \ref{lem:the space of convex} this means that $\phi_{\infty}=P(\chi_{F_{\Lambda}}\phi_{0})=((\chi_{F_{\Lambda}}\phi_{0})^{*})^{*}$
and hence $\phi_{\infty}^{*}=(\chi_{F_{\Lambda}}\phi_{0})^{*}.$ Moreover,
since the Legendre transform is a continuos operator on $\mathcal{C}_{\Lambda}$
it follows from the second Hopf formula that $\psi_{\infty}:=\lim_{t\rightarrow\infty}\psi_{t}=\phi_{\infty}^{*},$
which proves formula \ref{eq:limit of hj}. As a consequence $\psi_{\infty}(y)$
is locally the max of a finite number of affine functions (indeed,
since $F_{\Lambda}$ is locally finite and $\phi$ has quadratic growth
the sup defining $(\chi_{F_{\Lambda}}\phi_{0})^{*}(y)$ can, locally
wrt $y$, be taken over finitely points in $F_{\Lambda}$). Hence,
$\psi:=\psi_{\infty}$ is piecewise affine and hence so is $\phi_{\infty}$
(compare Remark \ref{rem:conv duality}).
\end{proof}
In particular, if $\psi_{0}(y)=|y|^{2}/2,$ then we can complete the
square and rewrite 
\[
\psi_{\infty}(y)=\frac{1}{2}|y|^{2}-\inf_{x\in F_{\Lambda}}\frac{1}{2}|x-y|^{2}.
\]
Accordingly the non-differentiability $S_{\psi_{\infty}}$ locus of
$\psi_{\infty}$ coincides with the subset of all points $y$ in $\R^{n}$
where the corresponding minimum is non-unique (compare Remark \ref{rem:conv duality}).
The latter set is the honeycomb like connected ($n-1)-$dimensional
piecewise linear manifold obtained as the union of the boundaries
of the open sets $\{O_{y}\}_{y\in F_{\Lambda}}$ consisting of points
in $\R^{n}$ for which $y$ is the unique closest point in $F_{\Lambda}.$
In the computational geometry literature the sets $O_{y}$ are called
Voronoi cells (attached to the point set $F_{\Lambda})$ and the corresponding
tessellation of $\R^{n}$ by convex polytopes is called the \emph{Voronoi
tessellation} (or Voronoi diagram) \cite{ok}. Similarly, the non-differentiability
locus $S_{\phi_{\infty}}$ of $\phi_{\infty}$ is the ($n-1)-$dimensional
stratum in the \emph{Delaunay tessellation} of $\R^{n}$ whose $0-$dimensional
stratum is given by the point set $F_{\Lambda}.$ The Delaunay tessellation
can be defined as the dual tessellation of the Voronoi tessellation,
in a suitable sense. For example, when $n=2$ this simply means that
$S_{\phi_{\infty}}$ is obtained by connecting any two points in $F_{\Lambda}$
which are neighbors in the corresponding Voronoi tessellation by a
segment \cite{ok}. 
\begin{rem}
Under suitable generality assumptions it is well-known that the corresponding
Delaunay tessellation consists of simplices giving a triangulation
of $F$ with remarkable optimality properties \cite{ok}.
\end{rem}
The previous proposition give a rigorous mathematical justification
of the Voronoi tessellations appearing in numerical simulations in
cosmology, which use periodic boundary conditions \cite{k-p-s-m,hsw,hsw2}:
for large times Voronoi polytopes form around points where $H$ has
its absolute minimum (the Voronoi polytopes in question are called
voids in the cosmology literature, since the mass in the universe
is localized on the shock locus $S_{\psi_{\infty}}$ between voids).
The dual Delaunay tessellation is also frequently used for the numerics
\cite{k-p-s-m,hsw,hsw2}. 
\begin{rem}
\label{rem:tropical} When $H$ has a unique minimum $x_{m}$ (modulo
$\Lambda),$ the corresponding convex piecewise affine function $\psi_{\infty}$
appears naturally in tropical geometry as a tropical theta function
with characteristics (in the case when $x_{m}$ and $\Lambda$ are
defined over the integers). The tropical subvariety defined by its
non-differentiability locus is called the tropical theta divisor and
seems to first have appeared in complex geometry in the compactification
of the moduli space of abelian varieties (see \cite{m-z} and references
therein). 
\end{rem}

\section{\label{sec:Application-to-Hele-Shaw}Application to Hele-Shaw type
flows}

\subsection{Background}

The Hele-Shaw flow was originally introduced in fluid mechanics in
the end of the 19th century to model the expansion of an incompressible
fluid of high viscosity (for example oil) injected at a constant rate
in another fluid of low viscosity (such as water) in a two dimensional
geometry. Nowadays the Hele-Shaw flow, also called\emph{ Laplacian
growth,} is ubiquitous in engineering, as well as in mathematical
physics where it appears in various areas ranging from diffusion limited
aggregation (DLA) to integrable systems (the dispersionless limit
of the Toda lattice hierarchy), random matrix theory and quantum gravity;
see \cite{va,g-v} and references therein.

To explain the general geometric setup, introduced in \cite{hs},
we let $X$ be a compact Riemann surface and fix a point $p$ (the
injection point) together with an area form $\omega_{0}$ of total
area one (whose density models the inverse permeability of the medium).
The classical situation appears when $X$ is the Riemann sphere and
$p$ is the point at infinity so that $X-\{p\}$ may be identified
with the complex plane $\C.$ A family of increasing domains $\Omega^{(\lambda)}$
with time parameter $\lambda\in[0,1]$ is said to be a classical solution
to the Hele-Shaw flow corresponding to $(p,\omega_{0})$ if $\Omega^{(0)}=\emptyset$
and the closure of $\Omega^{(\lambda)}$ is diffeomorphic to the unit-disc
in $\C$ for $\lambda>0,$ the point $p$ is contained in the interior
of $\Omega^{(\lambda)},$ the area grows linearly: 
\[
\int_{\Omega^{(\lambda)}}\omega_{0}=\lambda
\]
and the velocity of the boundary $\partial\Omega^{(\lambda)}$ equals
minus the gradient (wrt $\omega_{0})$ of the Green function $g_{p}$
for $\Omega^{(\lambda)}$ with a logarithmic pole at $p$ (i.e. Darcy's
law holds). Such a solution exists for $\lambda$ sufficiently small
(see \cite{hs} for the case when $\omega_{0}$ is real analytic and
\cite{r-wn2} for the general case). However, typically the boundary
of the expanding domains $\Omega^{(\lambda)}$ develop a singularity
for some time $\lambda<1$ and then changes its topology so that the
notion of a classical solution breaks down. Still, there is a well-known
notion of weak solution of the Hele-Shaw flow, defined in terms of
subharmonic envelopes (obstacles) and which exists for any $\lambda\in[0,1]$
(where $\Omega^{(1)}=X);$ see \cite{hs} and references therein.
In our notations the envelopes in question may be defined as 
\begin{equation}
\phi_{\lambda}:=\sup_{\phi\in PSH(X,\omega_{0})}\left\{ \phi:\,\phi\leq0,\,\,\,\phi\leq\lambda\log|z-p|^{2}+O(1)\right\} ,\label{eq:hs envelope}
\end{equation}
 which, for $\lambda$ fixed, is thus a restrained version of the
envelope $P_{\omega_{0}}(0)$ defined in Section \ref{sub:The-projection-operator},
where one imposes a logarithmic singularity of order $\lambda$ at
the given point $p.$ The \emph{weak Hele-Shaw flow }is then defined
as the evolution of the corresponding increasing non-coincidence sets:
\[
\Omega^{(\lambda)}:=\left\{ \phi_{\lambda}<0\right\} \subset X,
\]
(which thus is empty for $\lambda=0,$ as it should).  We will write
\[
X^{(\lambda)}:=X-\Omega^{(\lambda)}
\]
for the corresponding decreasing ``water domains''. When $\omega_{0}$
is real analytic it follows from the results in \cite{hs,sa} (applied
to the pull-back of $\omega_{0}$ to the universal covering $\tilde{X}$
of $X)$ that the boundary of $\Omega^{(\lambda)}$ is a piecewise
real analytic curve having a finite number of cusp and double points
(if moreover $\omega_{0}$ has negative Ricci curvature then the lifted
Hele-Shaw on $\tilde{X}$ exists for any $t>0).$
\begin{example}
\label{ex:HSF on riemann sphere}The classical situation in fluid
mechanics appears when $X$ is the Riemann sphere and $p$ is the
point at infinity, so that $X-\{p\}$ may be identified with the complex
plane $\C.$ Writing $\omega_{0}=dd^{c}\Phi_{0}$ in $\C$ (where
the condition $\Phi_{0}$ has logarithmic growth, since $\int\omega_{0}=1),$
the function $\phi_{\lambda}$ may be identified with the subharmonic
function $\Phi_{\lambda}:=\Phi_{0}+\phi_{\lambda}$ with the property
that $\Phi_{\lambda}=(1-\lambda)\log|z|^{2}+O(1)$ as $z\rightarrow\infty.$
Accordingly, $X^{(\lambda)}$ may, for $\lambda>0,$ be identified
with a decreasing family of compact domains in $\C.$ 
\end{example}

\subsection{A canonical regularization of the Hele-Shaw flow using the Kähler-Ricci
flows}

To make the link to the present setting of Kähler-Ricci flows we set

\begin{equation}
\theta=\omega_{0}-\delta_{p},\label{eq:theta in hs flow text}
\end{equation}
where $\delta_{p}$ denotes the Dirac measure at $p,$ which defines
a trivial cohomology class (this is thus a singular version of the
setting in Section \ref{sub:The trivial case} ). The corresponding
Kähler-Ricci flows will be defined as follows: first fixing a Kähler
form $\omega$ on $X$ we set 
\[
\theta_{\beta}:=\theta+\frac{1}{\beta}\mbox{Ric }\omega,
\]
 for a fixed Kähler form $\omega,$ i.e. by imposing the equation
\ref{eq:defining eq for f beta in terms of theta}. Moreover, we will
use $\omega_{0}$ as the initial data in the corresponding Kähler-Ricci
flows. We then get the following theorem saying that the corresponding
Kähler-Ricci flows concentrate, as $\beta\rightarrow\infty,$ precisely
on the complement $X^{(\lambda)}$ of $\Omega^{(\lambda)}$ (i.e.
on the ``water domain'') up to a time reparametrization:
\begin{thm}
\label{thm:krf conv to hsf on riemann s}Consider the non-normalized
Kähler-Ricci flow $\omega^{\beta}(t)$ with twisting current $\theta_{\beta}$
as above and initial condition $\omega_{0}.$ Then 
\begin{equation}
\lim_{\beta\rightarrow\infty}\omega^{(\beta)}(t)=1_{X-\Omega^{(\lambda(t))}}(t+1)\omega_{0},\label{eq:conv in thm krf to hsf on riemm}
\end{equation}
 weakly on $X,$ where $\Omega^{\lambda}$ is the weak Hele-Shaw flow
corresponding to $(p,\omega_{0})$ and $\lambda(t)=t/(t+1).$ Moreover,
\[
\sup_{X}\frac{\omega^{(\beta)}(t)}{\omega}\leq(t+1)\sup_{X}\frac{\omega_{0}}{\omega}.
\]
\end{thm}
\begin{rem}
If one instead let $\omega^{\beta}(t)$ denote the corresponding \emph{normalized}
Kähler-Ricci flow, which has total area $e^{-t}(=1-\lambda),$ then
the corresponding limiting measure is given by $1_{X-\Omega^{(\lambda(t))}}\omega_{0}$
and the last estimate above holds without the factor $(t+1).$ In
particular, in the canonical case, where $\omega$ is taken as $\omega_{0},$
setting $\eta_{t}:=\omega_{0}-\omega^{(\beta)}(t)$ then yields a
family of semi-positive forms of increasing area $1-e^{-t}$ concentrating
on the ``oil-domains'' $\Omega_{t}.$ 
\end{rem}
To prove the previous theorem we first need to make the link between
the envelopes \ref{eq:hs envelope} and the ones appearing in our
setting. To this end we introduce, as before, the potential $f$ of
$\theta$ (wrt the reference semi-positive form $\chi=0$ in $[\theta],$
satisfying
\[
\theta=dd^{c}f,
\]
which defines a lsc function $f:\,X\rightarrow]0,\infty]$ which is
smooth on $X-\{p\}$ and such that $-f$ has a logarithmic singularity
of order one at $p.$ 
\begin{lem}
\label{lem:The two envelopes are the same }The following holds 
\[
\phi_{\lambda}:=(1-\lambda)P_{\omega_{0}}\left(\frac{\lambda}{(1-\lambda)}f\right)-\lambda f,
\]
Equivalently, setting $t=\lambda/(1-\lambda)$ (i.e. $\lambda:=t/(t+1))$
gives 
\[
\Omega^{(\lambda)}:=\left\{ P_{\omega_{0}}(tf)<tf\right\} :=\Omega_{t}.
\]
\end{lem}
\begin{proof}
By a simple scaling argument it will be enough to prove that
\[
\phi_{\lambda}=P_{\omega_{0}(1-\lambda)}(\lambda f)-\lambda f.
\]
 But the latter identity follows immediately from the fact that a
given function $\phi\in PSH(X,\omega_{0})$ has a logarithmic pole
of order at least $\lambda$ at a point $p,$ i.e. it satisfies 
\[
\phi+\lambda f\leq C
\]
on $X$ iff the $\omega_{0}(1-\lambda)$-psh function $\phi+\lambda f$
on $X-\{p\}$ extends to a unique $\omega_{0}(1-\lambda)-$psh function
on all of $X$ (as follows from the basic local fact that a psh function
has a unique psh extension over an analytic subvariety, or more generally
over a pluripolar subset). 
\end{proof}
Finally, we need to extend Theorem \ref{thm:main general} to the
present setting. To this end we first recall that, by \cite[Theorem 3.2]{s-t},
there is, for $\beta$ fixed, a notion of weak Kähler-Ricci flows
on $X$ which applies to any twisting current $\theta$ which is smooth
away from a (suitable) divisor $D$ in $X.$ In particular, the result
applies to any current $\theta$ of the form 
\[
\theta=\theta_{0}-[E],
\]
 where $\theta_{0}$ is smooth and $[E]$ denotes the current of integration
along an effective divisor, i.e. 

\[
D=-E:=-\sum_{i}c_{i}E_{i}
\]
for $c_{i}>0$ and $E_{i}$ are irreducible hypersurfaces in $X.$
The result in \cite[Theorem 3.2]{s-t} yields a unique flow $\omega^{(\beta)}(t)$
of currents in $[\omega_{0}+t\theta]$ which are smooth on $X-D$
and such that the corresponding Kähler potentials are in $L^{\infty}(X)$
(as shown in \cite[Section 4.2]{egz} this flow coincides with the
unique viscosity solution constructed in \cite[Section 4.2]{egz}). 
\begin{thm}
\label{thm:singular version of main theorem}Let $\theta$ be a current
of the form $\theta=\theta_{0}-[E],$ with $\theta_{0}$ smooth and
$E$ an effective divisor. Then the conclusion in Theorem \ref{thm:main general}
still applies and the constant $C$ only depends on upper bounds on
$\theta_{0}$ (and the oscillation of its potential) and on the divisor
$E.$ Moreover, the sharp bounds in Theorem \ref{thm:sharp bounds}
still hold with $\theta$ replaced by $\theta_{0}.$ \end{thm}
\begin{proof}
We recall that the weak KRF defined in \cite[Theorem 3.2]{s-t} is
constructed by approximating $\theta$ with a suitable sequence $\theta_{\epsilon}$
of smooth forms. In the present setting this can be done so that $\theta_{\epsilon}\leq C\omega$
and $\theta_{\epsilon}$ converges to $\theta$ in $\mathcal{C}_{loc}^{\infty}(X-E).$
Indeed, decomposing $f=f_{0}+f_{E}$ in terms of potentials for $\theta_{0}$
and $-[E],$ respectively, we have that up to a smooth function $f$
can be written as $-\log\left\Vert s_{E}\right\Vert ^{2},$ where
$s_{E}$ is a holomorphic section of the line bundle $\mathcal{O}(E)$
cutting out $E$ and $\left\Vert \cdot\right\Vert $ is a fixed smooth
Hermitian metric on $\mathcal{O}(E).$ Then the form $\theta_{\epsilon}$
is simply obtained by replacing $\log\left\Vert s_{E}\right\Vert ^{2}$
with $\log(\left\Vert s_{E}\right\Vert ^{2}+\epsilon).$ The proof
of the theorem then follows immediately from Theorem \ref{thm:main general}
applied to $\theta_{\epsilon}$ by noting that that $P(f)\leq\sup_{X}f_{0}+P(f_{E}),$
where the second term thus only depends on the divisor $E,$ as desired
(and is finite, by Lemma \ref{lem:proj of plus infity}).\end{proof}
\begin{example}
\label{ex:krf on sphere}Coming back to the classical setting when
$E$ is the point $p$ and $X-\{p\}=\C$ considered in the previous
example, the density $\rho^{(\beta)}(t)$ wrt Lebesgue measure on
$\C$ of the Kähler form $\omega^{(\beta)}(t)$ on $X-\{p\}$ is a
solution of the following logarithmic diffusion equation for the smooth
and strictly positive probability densities $\rho(t)$ on $\C$ 
\[
\frac{\partial\rho(t)}{\partial t}=\frac{1}{\pi\beta}\partial_{\bar{z}}\frac{\partial_{z}\rho(t)}{\rho(t)}+\rho_{0}+O(\frac{1}{\beta}),\,\,\,\,\rho(0)=\rho_{0},
\]
where the last term is equal to $\frac{1}{\beta}\Delta\log\rho_{0}(t)$
(but it could be removed at the expense of slightly worse estimates
in $t$ and $\beta).$ The equivalence between Ricci flow on Riemann
surfaces and logarithmic diffusion is well-known \cite{v-e-r}, but
as far as we know the limit $\beta\rightarrow\infty$ has not been
investigated before.
\end{example}

\subsection{Monge-Ampère growth}

There is also a natural higher dimensional generalization of the Hele-Shaw
flow/Laplacian growth on a compact Kähler manifold $(X,\omega_{0})$
where the higher dimensional viscous ``fluid'' is injected along
a given effective divisor $E$ on $X.$ Indeed, one simply defines
$\phi_{\lambda}$ as before, but imposing a singularity of order $\lambda$
along $E$ (i.e. $z-p$ is in formula \ref{eq:hs envelope} replaced
by a local defining equation for $E).$ Then one obtains a sequence
of increasing domains $\Omega_{\lambda}$ as before for which the
name \emph{Monge-Ampère growth} was proposed in \cite{berm1b}. The
terminology is motivated by the fact that $\Omega_{\lambda}$ can
be characterized as the solution of a free boundary problem for the
complex Monge-Ampère operator on $(X,\omega_{0})$ with singular obstacle
$\lambda f$ (see Remark\ref{rem:free}), where $f$ is defined by
\[
\theta=\omega_{0}-[E],\,\,\,\,\theta=dd^{c}f,
\]
 as before. By the recent results in \cite{rwn3}, for $\lambda$
sufficiently small, $\Omega_{\lambda}$ is diffeomorphic to a ball
(and admits a regular foliation, transversal to $E,$ by holomorphic
discs along which $\phi_{\lambda}$ is $\omega_{0}-$harmonic). 

Now, by Theorem \ref{thm:singular version of main theorem}, the volume
forms $\omega_{\beta}^{n}(t)$ of the Kähler-Ricci flows with twisting
form $\theta$ as above concentrate on $X^{(\lambda)}(:=X-\Omega^{(\lambda(t))}):$
\[
\lim_{\beta\rightarrow\infty}\omega_{\beta}^{n}(t)=\frac{1_{X^{(\lambda)}}\omega_{0}^{n}}{\int_{X^{(\lambda)}}\omega_{0}^{n}}
\]
with uniform upper bounds on the normalized Kähler forms $\omega_{\beta}(t)/(t+1)$
on $X-E,$ as before (in this setting $\int_{X^{(\lambda)}}\omega_{0}^{n}=[\omega_{0}-\lambda(t)E]^{n}).$ 
\begin{example}
In the case when $X=\P^{n}$ equipped with a Kähler form $\omega_{0}$
of unit volume and $E$ is the hyperplane at infinity the corresponding
sets $X(t)$ yield, for $t>0,$ a decreasing family of compact domains
in $\C^{n}$ of volume $1/(t+1)^{n}.$ \end{example}
\begin{rem}
\label{rem:hsf gives weak geod} As shown in \cite{r-wn1} performing
a Legendre transform of $\phi_{\lambda}$ with respect to $\lambda$
produces a weak geodesic ray $\hat{\phi}_{\tau}$ in the space of
Kähler metrics (compare Remark \ref{rem:r-z}). Moreover, topology
change in the corresponding Hele-Shaw flow $\Omega^{(\lambda)}$ corresponds
(in a certain sense) to singularities of the geodesic $\hat{\phi}_{\tau}$
\cite{rwn4,rwn5}. In a nutshell, this stems from the the fact (shown
in \cite{r-wn1}) that $\Omega^{(\lambda)}=\{h<\lambda\}$ where $h(x):=\frac{d\hat{\phi}_{\tau}}{d\tau}|_{\tau=0^{+}}.$ 
\end{rem}

\section{The case of twisting currents with merely continuous potentials \label{sub:An-extension-of}}

Without loss of generality we may and will in this section, assume
that $\varphi_{0}=0.$ As will be next explained the weak convergence
in Theorem \ref{thm:main intro} can be extended to any twisting form
(or rather current) with continuous potentials. 

To illustrate this we start with the case $n=1$ and assume that $\frac{\text{1}}{\beta}c_{1}(K_{X})+[\theta_{\beta}]$
is trivial, i.e. that the non-normalized KRF preserves the initial
cohomology class. To simplify the notation we will drop the subscript
$\beta$ in the notation $f_{\beta}$ for the corresponding twisting
potential. 
\begin{prop}
Assume that $n=1$ and $f$ is Hölder continuous. Then there is a
unique solution $\varphi^{(\beta)}(t)$ to the corresponding non-normalized
KRF which is in $C^{2,\alpha}(X)$ for some $\alpha>0.$\end{prop}
\begin{proof}
In the following $\beta$ will be fixed and we will not pay attention
to the dependence on $\beta.$ First assume that $f$ is smooth. Differentiating
the non-normalized KRF wrt $t$ reveals that $d\varphi^{(\beta)}(t)/dt$
evolves by the heat equation for the metric $\omega_{\beta}(t)$ and
hence, by the parabolic maximum principle, $|d\varphi^{(\beta)}(t)/dt|\leq C,$
where the constant only depends on $\sup_{X}|f|.$ The defining equation
for the KRF then gives that $C'^{-1}\leq\omega_{\beta}(t)\leq C'$
for a positive constant $C'$ only depending on $\sup_{X}|f|.$ But
then applying the parabolic Krylov-Safonov Hölder estimate to the
heat equation wrt $\omega_{\beta}(t)$ gives that there exists a Hölder
exponent $\alpha'$ such that $\left\Vert d\varphi^{(\beta)}(t)/dt\right\Vert _{C^{\alpha'}}\leq C''.$
Using again the defining equation for $\varphi^{(\beta)}(t)$ we deduce
that, $1+\Delta_{\omega}\varphi^{(\beta)}(t)=e^{\beta g_{\beta}(t)},$
where the Hölder norm of $g_{\beta}(t)$ is under control, for some
Hölder exponent. But then the proof is concluded by invoking the classical
Schauder estimates for the Laplacian $\Delta_{\omega}$ and approximating
$f$ with smooth functions (note that the limit of the approximate
solutions is unique, by the comparison principle).
\end{proof}
Given a twisting potential $f$ we denote by $P_{t}^{(\beta)}f$ the
solution of the corresponding KRF at time $t$ and set $P_{t}f:=P_{\omega_{0}}(tf).$
\selectlanguage{english}%
\begin{lem}
\label{lem:The-operator-P beta}\foreignlanguage{american}{The operator
$P_{t}^{(\beta)}$ is increasing, i.e. if $f\leq g,$ then $P_{t}^{(\beta)}f\leq P_{t}^{(\beta)}g.$
Moreover, $P_{t}^{(\beta)}(f+c)=P_{t}^{(\beta)}(f)+ct$ for any $c\in\R$
and hence
\begin{equation}
\left\Vert P_{t}^{(\beta)}f-P_{t}^{(\beta)}g\right\Vert _{L^{\infty}(X)}\leq t\left\Vert f-g\right\Vert _{L^{\infty}(X)},\label{eq:contr prop of P beta}
\end{equation}
and similarly for the operator $P_{t}.$}\end{lem}
\selectlanguage{american}%
\begin{proof}
The increasing property follows directly from the comparison principle
and the scaling property from the very definitions of the flows.\end{proof}
\begin{thm}
Let $X$ be a Riemann surface endowed with the twisting current $\theta=dd^{c}f,$
where $f$ is Hölder continuous. Then the corresponding non-normalized
KRFs $\omega_{\beta}(t)$ defines a family of Hölder continuous Kähler
metrics satisfying the weak convergence in Theorem \ref{thm:main intro},
as $\beta\rightarrow\infty$ (more precisely, the convergence holds
in $C^{0}(X))$ on the level of Kähler potentials).\end{thm}
\begin{proof}
In the following $t$ will be fixed once and for all. Let $f_{\epsilon}$
be a family of smooth functions such that $\left\Vert f_{\epsilon}-f\right\Vert _{\infty}\leq\epsilon.$
By the previous lemma 
\[
\left\Vert P_{t}^{(\beta)}f-P_{t}f\right\Vert _{L^{\infty}(X)}\leq\left\Vert P_{t}^{(\beta)}f_{\epsilon}-P_{t}f_{\epsilon}\right\Vert _{L^{\infty}(X)}+2\epsilon t.
\]
Hence, letting first $\beta\rightarrow\infty$ (using Theorem \ref{thm:main general})
and then $\epsilon\rightarrow0$ concludes the proof. 

Of course, even if $\omega_{\beta}(t)$ is bounded for a fixed $\beta$
the limiting current $\omega_{\infty}(t)$ will, in general, not be
bounded unless $f$ has a bounded Laplacian. The previous theorem
also holds when $f$ is assumed to be merely continuous, but then
the corresponding evolution equations have to be interpreted in a
generalized sense. More generally, when $f$ is continuous and the
dimension $n$ of $X$ is arbitrary the corresponding KRFs are well-defined
in the sense of viscosity solutions and satisfy the comparison principle,
by \cite{egz}. Accordingly, the $C^{0}-$convergence in the previous
theorem still holds. However, even if $f$ is Hölder continuous it
does not seem to follow, in general, from existing regularity theory
that $\omega_{\beta}(t)$ is even bounded, for $\beta$ fixed. 
\end{proof}

\subsection{An outlook on random twistings\label{rem:random potential}}

Hölder continuous potentials $f$ appear naturally when $f$ is taken
to be an appropriate random Gaussian function. For example, in the
setting described in Section \ref{sub:Relation-to-the-krf}, when
$n=1$ and $X=\R/\Z+i\R/\Z$ and the potential $f$ is assumed invariant
along the imaginary direction, we can identify the potential $f$
with a $1-$periodic function $f(x)$ on $\R$ and expand $f(x)$
in a Fourier series:

\[
f(x)=\sum_{k\in\Z}A_{k}\cos(2\pi kx)+B_{k}\sin(2\pi kx).
\]
Taking the coefficients $A_{k}$ and $B_{k}$ to be independent Gaussian
random numbers with mean zero and variance proportional to $k^{-3-2h},$
for a given number $h\in[-1,1],$ it is well-known that $f$ is almost
surely in the Hölder class $C^{1,h}.$ Indeed, the derivative $f'$
is a Brownian fractional bridge, whose sample paths are well-known
to be almost surely in $C^{h}$ (recall that a Brownian bridge is
defined as a Brownian motion $B$ conditioned by $B(0)=B(1)$ and
similarly in the fractional case, with $h=1/2$ corresponding to ordinary
Brownian motion). The corresponding limiting convex envelopes $\phi_{t}(x)$
have been studied extensively in the mathematical physics literature
in the setting of Burger's equation and the adhesion model, where
$f'$ represents the random initial velocity function (compare Section
\ref{sec:Relations-to-viscosity}). According to a conjecture in \cite{saf},
for any fixed positive time $t,$ the support $X_{t}$ of the distribution
second derivative of the corresponding random function $\phi_{t}(x)$
on $\R$ is almost surely of Hausdorff dimension $h$ when $h\in[0,1]$
and $0$ when $h\in[-1,0]$ (which, when $h=1$ is consistent with
the uniform bound in Theorem \ref{thm:main intro} and formula \ref{eq:point wise formula for MA of env in text}
which, in this real setting, holds as long as $f\in C^{1,1}).$ See
\cite{gi} for the case when $h=-1/2$ and \cite{sin} for a proof
of the conjecture in the case $h=1/2$ in a non-periodic setting.

In view of the connections to the Kähler-Ricci flow and the Hele-Shaw
flow exhibited in Sections \ref{sec:Relations-to-viscosity} and \ref{sec:Application-to-Hele-Shaw}
it would be interesting to extend this picture to any complex manifold,
or at least to Riemann surfaces. For example, in the latter case one
would, at least heuristically, get conformally invariant processes
of random metrics $\omega_{\beta}(t)$ by taking $f$ to be a Gaussian
free field on $X.$ Heuristically, this means that $f$ is taken as
random function in the corresponding Dirichlet Hilbert space $H^{1}(X)/\R.$
However, the situation is complicated by the fact that, almost surely,
$f$ only exists as a distribution in a certain Banach completion
of $H^{1}(X)/\R.$ \cite{sh}. On the other hand the formal random
measure appearing in the static version of the non-normalized KRF,
i.e. in the Laplace equation 
\[
\omega_{0}+dd^{c}\varphi_{\beta}(t)=e^{-\beta f}\omega_{0}
\]
appears as the Liouville measure of quantum gravity and has been rigorously
defined, for $\beta\in]0,2[$ in \cite{d-s} using a regularization
procedure. But as far as we know the corresponding stochastic parabolic
problem has not been investigated.

\section{\label{sec:The-gradient-flow}The gradient flow picture (an outlook)}

In this section we introduce a complementary point of view on the
convergence result in Theorem \ref{thm:main intro}, which in particular
leads to a gradient flow type realization of the limiting flows. We
also indicate the relations to stochastic interacting particle system
and the thermodynamical formalism introduced in \cite{berm4}. A more
complete picture will appear in a separate publication.

Let $X$ be a compact complex manifold endowed with a Kähler class
$T\in H^{2}(X,\R)$ and denote by $\mathcal{K}(X,T)$ the space of
all Kähler metrics $\omega$ in $T.$ Up to a trivial scaling we may
and will assume that $T^{n}=1.$ Fixing a reference Kähler metric
$\omega_{0}$ in $T$ we will identify $\mathcal{K}(X,T)$ with the
corresponding space $\mathcal{H}(X,\omega_{0})/\R$ of Kähler potentials
(modulo constants). Occasionally, it will also be convenient to identify
a Kähler metric with its normalized volume form, using the Calabi-Yau
map 
\begin{equation}
\omega\mapsto\mu:=\omega^{n},\,\,\,\mathcal{K}(X,T)\rightarrow\mathcal{P}(X)\label{eq:c-y iso}
\end{equation}
which induces an isomorphism between the space $\mathcal{K}(X,T)$
of all Kähler metrics $\omega$ in $T$ and the subspace $\mathcal{P}^{\infty}(X)$
of all volume forms in the space $\mathcal{P}(X)$ of all probability
measures on $X$ \cite{y}.

In this section we will focus on the case when the cohomology class
is not moving under the corresponding\emph{ normalized} KRF (as in
Section \ref{sub:The-case-when convex envelopes}), which equivalently
means that 
\[
T=\frac{1}{\beta}c_{1}(K_{X})+[\theta_{\beta}].
\]
(however, see Section \ref{sub:The-nonnormalized-setting} for the
non-normalized setting). As before we denote by $f$ the potential
of $\theta:$ 
\[
\theta=\omega_{0}+dd^{c}f
\]
 Occasionally it will be convenient to pass between Kähler potentials
$\varphi$ relative to $\omega_{0}$ and Kähler potentials $u$ relative
to $\theta$ by setting 
\[
u:=\varphi-f
\]
 ensuring that $\omega_{\varphi}=\theta_{u}.$

\subsection{The twisted Kähler-Ricci flow as a gradient flow}

We recall that the gradient flow of a smooth function $F$ on a Riemannian
manifold $Y$ is the flow defined by 
\[
\frac{dy(t)}{dt}=-(\nabla F)(y(t)),\,\,\,\,\,y(0)=y_{0}
\]
 where $\nabla$ denotes the gradient wrt the given Riemannian metric
on $Y.$ In our infinite dimensional setting we equip the space $\mathcal{K}(X,T)$
with the Riemannian metric defined as follows: 
\begin{equation}
\left\langle u,u\right\rangle _{\varphi}:=n\int_{X}du\wedge d^{c}u\wedge\omega_{\varphi}^{n-1},\label{eq:def of Dirichlet type metric}
\end{equation}
where the tangent space of $\mathcal{K}(X,T)$ at $\varphi$ has been
identified with $C^{\infty}(X)/\R$ in the usual way (i.e. using the
standard affine structure). In other words, $\left\langle u,u\right\rangle _{\varphi}$
is the $L^{2}-$norm of the gradient of $u$ wrt the Kähler metric
$\omega_{\varphi}.$

Next, we recall that the $\theta-$twisted (and $\beta-$normalized)
version of Mabuchi's K-energy functional on $\mathcal{H}(X,\omega_{0})/\R$
is defined by specifying its differential, viewed as a measure valued
operator: 
\[
-(\delta M_{\theta}^{(\beta)})_{|\varphi}=\left(\frac{1}{\beta}\mbox{Ric \ensuremath{\omega_{\varphi}-\theta}}\right)\wedge\omega_{\varphi}^{n-1}-C\omega^{n},
\]
 where $C$ is the cohomological constant ensuring that rhs above
integrates to zero.
\begin{prop}
\label{prop:krf as gradient-flow}The gradient flow of the twisted
K-energy $M_{\theta_{\beta}}^{(\beta)}$ on $\mathcal{K}(X,T)$ coincides
with the normalized KRF with twisting form $\theta_{\beta}.$ \end{prop}
\begin{proof}
It will be convenient to use the ``thermodynamical formalism'' \cite{berm4}
in order to identify $M_{\theta}^{(\beta)}$ with a free energy type
functional $F_{\beta}$ on $\mathcal{P}(X):$

\[
M_{\theta}^{(\beta)}(\varphi)=F_{\beta}(\mu),\,\,\,\,\mu=\omega_{\varphi}^{n}
\]
 where 
\begin{equation}
F_{\beta}(\mu)=E_{\theta}(\mu)+\frac{1}{\beta}H_{\mu_{0}}(\mu),\label{eq:free energy}
\end{equation}
and where $E_{\theta}(\mu)$ is the pluricomplex energy of $\mu$
relative to $\theta$ and $H_{\mu_{0}}(\mu)$ is the entropy of $\mu$
relative to a certain fixed normalized volume form $\mu_{0}$ determined
by $\theta$ and $\omega_{0}$ (see \cite{berm4}).\footnote{From a thermodynamical point of view $F_{\beta}$ is the Gibbs free
energy at inverse temperature $\beta$ for a system with internal
energy $E.$} Next, we make the following general observation: if $M(\varphi)=F(\mu)$
then the following relation holds 
\begin{equation}
(\nabla M)_{|\varphi}=-(\delta F)_{|\mu},\label{eq:gradient of functional wrt Dirichlet metric}
\end{equation}
between the gradient $(\nabla M)_{|\varphi}\in C^{\infty}(X)$ of
$M$ at $\varphi$ wrt the Dirichlet metric and the differential $(\delta F_{|\mu})$
at $\mu,$ identified with a function on $X$ (using the standard
integration pairing between functions and measures). Indeed, if $\mu(t)$
is a curve such that $\mu(0)=\mu$ then, by the very definition of
$\delta F_{|\mu},$ we have 
\[
\int_{X}\delta F_{|\mu}\frac{d\mu(t)}{dt}_{|t=0}:=\frac{dF(\mu(t))}{dt}_{|t=0}.
\]
In particular, if $\mu(t)=\omega_{\varphi(t)}^{n}$ then setting $u:=\delta F_{|\mu}$
gives 
\[
\frac{dM(\varphi(t))}{dt}\Big|_{t=0}=\frac{dF(\mu(t))}{dt}\Big|_{t=0}=n\int_{X}u\,dd^{c}\left(\frac{d\varphi(t)}{dt}\right)\wedge\omega_{\varphi(0)}^{n-1}=
\]
\[
=-n\int_{X}du\wedge d^{c}\left(\frac{d\varphi(t)}{dt}\Big|_{t=0}\right)\wedge\omega_{\varphi(0)}^{n-1},
\]
which, by definition, equals $-\left\langle u,\frac{d\varphi(t)}{dt}\Big|_{t=0}\right\rangle _{\varphi},$
proving formula \ref{eq:gradient of functional wrt Dirichlet metric}.
Finally, as shown in \cite{berm4} we have 
\[
(\delta E_{\theta})_{|\mu}=-(\varphi-f),\,\,\,(\delta H_{\mu_{0}})_{|\mu}=\log\left(\frac{\mu}{\mu_{0}}\right)
\]
and hence the gradient flow of $M_{\theta}^{(\beta)}(\varphi)$ is
given by 
\[
\frac{d\varphi(t)}{dt}=\frac{1}{\beta}\log\left(\frac{\omega_{\varphi(t)}^{n}}{\mu_{0}}\right)-(\varphi(t)-f),
\]
 as desired (modulo constants). \end{proof}
\begin{rem}
The metric \ref{eq:def of Dirichlet type metric} seems to first have
appeared in \cite{cal}, where it is attributed to Calabi and called
Calabi's gradient metric (not to be confused with another metric usually
referred to as the Calabi metric obtained by replacing the gradient
of $u$ by the Laplacian of $u).$ The metric \ref{eq:def of Dirichlet type metric}
was further studied in \cite{c-k,cpz} where it is called the Dirichlet
metric. See also \cite{fi} where the metric \ref{eq:def of Dirichlet type metric}
appears from a symplecto-geometric point of view. In Section \ref{sub:Relations-to-the}
below we will give a new interpretation of the metric \ref{eq:def of Dirichlet type metric},
motivated by probabilistic considerations. The relation between the
KRF and gradient flows wrt the Dirichlet metric first appeared in
\cite{c-z} (in the non-twisted setting). 
\end{rem}

\subsection{The zero-temperature limit $\beta\rightarrow\infty$}

We denote by $\varphi_{t}$ the following curve of functions in $PSH(X,\omega_{0}):$
\[
\varphi(t):=P_{\omega_{0}}\left(e^{-t}\varphi_{0}+(1-e^{-t})f\right)
\]
and by $\mu(t)$ the corresponding curve of probability measures on
$X.$ Using that the latter curve arises as the large $\beta-$limit
of the corresponding normalized Kähler-Ricci flows (by Theorem \ref{thm:main general})
Proposition \ref{prop:krf as gradient-flow} implies that $E_{\theta}(\mu(t))$
is decreasing. But, in fact, a direct argument reveals that it even
strictly decreasing away from a minimizer: 
\begin{prop}
The pluricomplex energy $E_{\theta}(\mu(t))$ is decreasing wrt $t$,
and strictly decreasing unless $\mu(t)$ reaches a minimizer. Moreover,
$E_{\theta}(\mu(t))$ converges to the infimum of $E_{\theta}$ over
$\mathcal{P}(X)$ and $\mu(t)$ converges to the corresponding minimizer.\end{prop}
\begin{proof}
First of all observe that, by Prop \ref{prop:krf as gradient-flow}
$M_{\theta}^{(\beta)}$ is decreasing along $\varphi^{(\beta)}(t)$
for $\beta$ fixed. But, by the uniform bound on the Laplacians we
have that $H(\mu^{(\beta)}(t))\leq C$ and $E_{\theta}(\mu_{t})\rightarrow E_{\theta}(\mu)$
as $\beta\rightarrow\infty$ and hence $E_{\theta}(\mu(t))$ is also
decreasing, as desired. Alternatively a direct proof using envelopes
can be given as follows, which also includes strict monotonicity.
Recall that $\mu(t)=(\omega_{0}+dd^{c}\varphi_{t})^{n}$ and $E_{\theta}$
(acting on the level of potentials) is defined by 
\[
E_{\theta}(\varphi)=E(\varphi)-\int_{X}\varphi\MA(\varphi)+\int_{X}f\MA(\varphi),
\]
where $E=\AM$ is the Aubin-Mabuchi energy. We denote $\varphi_{t}=P_{\omega}(e^{-t}\varphi_{0}+(1-e^{-t})f)$
and we assume for simplicity that $\varphi_{0}=0$. Fix $t\geq0,s>0$.
By basic properties of the Aubin-Mabuchi functional we have 
\[
\AM(\varphi)-\AM(\psi)\leq\int_{X}(\varphi-\psi)\MA(\psi).
\]
We will use the $I$ functional in \cite{BBEGZ} : $I(u,v)=\int_{X}(u-v)(\MA(v)-\MA(u))\geq0$.
Using the formula of $E_{\theta}$ and the inequality above we can
write 
\begin{eqnarray}
E_{\theta}(\varphi_{t+s})-E_{\theta}(\varphi_{t})\leq I(\varphi_{t+s},\varphi_{t})+\int_{X}(f-\varphi_{t})(\MA(\varphi_{t+s}-\MA(\varphi_{t})).\label{eq:monotone of E theta 0}
\end{eqnarray}
We claim that 
\[
\int_{X}(f-\varphi_{t})(\MA_{\omega}(\varphi_{t+s})-\MA_{\omega}(\varphi_{t}))\leq-\lambda I(\varphi_{t+s},\varphi_{t}),
\]
where $\lambda=\frac{e^{s}}{e^{s}-1}$. Indeed, let $\Omega_{t}:=\{\varphi_{t}<(1-e^{-t})f\}$
be the non-coincidence set. By the monotonicity result we know that
$\Omega_{t}\subset\Omega_{t+s}$. It suffices to prove that 
\begin{equation}
\int_{X}(f-\lambda\varphi_{t+s}+(\lambda-1)\varphi_{t}))(\MA_{\omega}(\varphi_{t+s})-\MA_{\omega}(\varphi_{t}))\leq0.\label{eq:monotone of E theta 1}
\end{equation}
The integrand is non-negative thanks to Proposition \ref{prop:The-sup-coincides with projection of linear curve},
and it vanishes out side $\Omega_{t+s}$. As $\MA_{\omega}(\varphi_{t+s})$
vanishes in $\Omega_{t+s}$, the inequality (\ref{eq:monotone of E theta 1})
follows. Now (\ref{eq:monotone of E theta 0}) and (\ref{eq:monotone of E theta 1})
give that 
\[
E_{\theta}(\varphi_{t+s})-E_{\theta}(\varphi_{t})\leq\frac{-1}{e^{s}-1}I(\varphi_{t+s,}\varphi_{t}).
\]
Finally, if $E_{\theta}(\varphi_{t+s})=E_{\theta}(\varphi_{t})$ then
we must have $\varphi_{t+s}=\varphi_{t}$, as $I$ is non-degenerate.
The next lemma shows that the flow is stationary from $t$. In fact,
if $\{f=0\}$ has Lebesgue measure zero then $E_{\theta}$ is strictly
decreasing.\end{proof}
\begin{lem}
Denote by $\varphi_{t}=P_{\omega}(e^{-t}\varphi_{0}+(1-e^{-t})f)$.
If $\varphi_{t}=\varphi_{s}$ for some $0\leq t<s$ then $\varphi_{t+h}=\varphi_{t}$
for all $h\geq0$.\end{lem}
\begin{proof}
Again, for simplicity we assume that $\varphi_{0}=0$. As $\MA(\varphi_{t})=\MA(\varphi_{t+s})$,
the measure is concentrated on the set $\{\varphi_{t}=(1-e^{-t})f=(1-e^{-t-s})f\}$,
which equals $\{\varphi_{t}=f=0\}$. Now, as $\MA(\varphi_{t})(\varphi_{t}<0)=0$
the domination principle gives that $\varphi_{t}\geq0$. It follows
that $f\geq0$, and hence $\varphi_{t}$ is increasing in $t$. But
again we have that $\MA(\varphi_{t})$ vanishes on $\{\varphi_{t}<\varphi_{t+h}\}$.
To see this we note that the measure is supported on the coincidence
set and on this set, by the monotonicity property (Proposition \ref{prop:The-sup-coincides with projection of linear curve}),
$\varphi_{t+h}\leq e^{-s}\varphi_{t}+(1-e^{-s})f=0$. Thus the domination
principle again yields $\varphi_{t}\geq\varphi_{t+h}$, hence equality
holds. 
\end{proof}
In the light of the previous results one would expect that the curve
$\mu{}_{t}$ arises as a gradient flow of $E_{\theta}$ (in the sense
of metric spaces \cite{a-g-s}) when $\mathcal{P}(X)$ is equipped
with the metric induced by the metric \ref{eq:def of Dirichlet type metric}
(under the Calabi-Yau isomorphism). However, in order to make this
precise one has to deal with several technical problems related to
the geometry of the metric completion of the space $\mathcal{K}(X,T)$
equipped with the Dirichlet metric above, that we leave for the future\footnote{Unfortunately, when $n>1,$ the corresponding metric geometry appears
to be more complicated - from the point of view of gradient flows
- than the case of the Mabuchi-Semmes-Donaldson metric on $\mathcal{K}(X,T)$
whose metric completion has non-positive sectional curvature and where
the corresponding gradient flow of the K-energy functional yields
a weak version of the Calabi flow \cite{str,bdl}.}. Here we just formulate a precise result in the case when $n=1$
where the metric \ref{eq:def of Dirichlet type metric} coincides
with the classical Dirichlet norm on the Riemann surface $X.$ To
this end we denote by $H^{1}(X)/\R$ the Sobolev quotient space obtained
by completing the Dirichlet norm on $C^{\infty}(X)/\R.$ Since $H^{1}(X)/\R$
is a Hilbert space there is a classical notion of gradient flows of
lsc convex functionals on $H^{1}(X)/\R$ which we briefly recall.
Given a lsc convex function $F$ on a Hilbert space $H:$ a curve
$v(t),$ which is absolutely continuous as a map from $]0,\infty[$
to $H,$ is said to be the gradient flow of $F$ emanating from $v_{0}$
if $v(t)\rightarrow v_{0}$ as $t\rightarrow0$ and for almost any
$t$ 
\begin{equation}
\frac{dv(t)}{dt}\in\partial H_{|v(t)},\label{eq:subgradient relation}
\end{equation}
where the rhs above denotes the subgradient of $H$ at $v(t).$ There
are also other equivalent definitions. For example, the differential
inclusion \ref{eq:subgradient relation} may be replaced by the the
following Evolutionary Variational Inequalities (EVI): for any given
$w\in H$ 
\[
\frac{d}{dt}\frac{1}{2}\left\Vert v(t)-w\right\Vert ^{2}+F(v(t))-F(w)\leq0.
\]
In turn, this is equivalent to $v(t)$ arising as a limit of a Minimizing
Movement, i.e. a variational form of the backward Euler discretization
scheme. The virtue of the latter two characterizations is that they
can be formulated when the Hilbert space $H$ is replaced by a general
metric space (in particular, the corresponding weak gradient flows
always exist when $F$ is a lsc convex function on a complete metric
space with non-positive sectional curvature; see \cite{a-g-s} and
references therein). 
\begin{thm}
\label{thm:hilb gradient flow in normalized setting}Let $X$ be a
Riemann surface endowed with a smooth two-form $\theta.$ Equip the
Sobolev space $H^{1}(X)/\R$ with the classical Dirichlet metric and
consider the following lower semi-convex functional $F$ on $H^{1}(X)/\R$\textup{:
}\textup{\emph{
\[
F(u):=\frac{1}{2}\int_{X}du\wedge d^{c}u
\]
for $u$ in the convex subset $C_{\theta}$ of $H^{1}(X)/\R$ defined
by the condition $dd^{c}u+\theta\geq0$ and let $F=\infty$ on the
complement of $C_{\theta}.$ Denote by $u(t)$ the solution of the
gradient flow of $F$ emanating from a given element $u_{0}\in C_{\theta}\cap C^{\infty}(X).$
Then $dd^{c}u(t)+\theta$ coincides with the curve $\mu(t)$ of probability
measures on $X$ defined by the envelope construction above.}}\end{thm}
\begin{proof}
Fix a normalized volume form $dV$ on $X$ and let $F_{\beta}$ be
the corresponding free energy functional (formula \ref{eq:free energy})
on $H^{1}(X)/\R:$ $F_{\beta}(u)=F(u)+H_{dV}(dd^{c}u(t)+\theta)/\beta.$
This is a convex lsc functional and hence its gradient flow $u^{(\beta)}(t)$
emanating from the given element $u_{0}$ is well-defined. It then
follows from well-known stability results that $u^{(\beta)}(t)\rightarrow u(t)$
in $H^{1}(X)/\R.$  But, by the uniqueness of weak gradient flows
in Hilbert spaces, $dd^{c}u_{\beta}(t)+\theta$ coincides with the
curve of Kähler forms defined by the KRF (compare Proposition \ref{prop:krf as gradient-flow})
and hence the proof is concluded by invoking Theorem \ref{thm:main general}.
Alternatively, a direct proof can be given as follows: by Proposition
\ref{prop:krf as gradient-flow} and the convexity of $F_{\beta}$
the curve $u^{(\beta)}(t)$ satisfies the Evolutionary Variational
Inequalities wrt $F_{\beta}.$ Then, passing to the limit and using
Theorem \ref{thm:main general}, reveals that $u(t)$ also satisfies
the latter inequalities wrt $F$, which, as recalled above, is equivalent
to $u(t)$ being the gradient flow of $F.$ 
\end{proof}
Note that one virtue of the EVI formulation of the gradient flow in
the previous theorem (used in the end of the proof) is that the gradient
flow is intrinsically defined on the convex subset $C_{\theta}$ (which
is a Euclidean complete metric space, but not a Hilbert space). 
\begin{rem}
The solution of the gradient flow in the previous theorem can be given
by the following description purely in terms of the geometry of the
Hilbert space $H:=H^{1}(X)/\R.$ Let $F$ be half the squared Hilbert
space norm $h^{2}/2$ on $H$ restricted to a given compact convex
subset $C$ (which does not contain the origin) and then extended
by $\infty$ to all of $H-C,$ i.e. $F=h^{2}/2+\chi_{C},$ where $\chi_{C}$
is the indicator function of $C.$ Given an initial point in $C$
the gradient flow of $h^{2}/2$ is an affine curve $v(t)$ which leaves
the space $C$ after a finite time. However, replacing $v(t)$ with
$P(v(t)),$ where $P(v)$ is the projection of $v$ onto $C,$ gives
the weak gradient flow of $F,$ which does stay in $C.$ To be more
precise: $P(v)$ is the point in $C$ which is closest to $v$ wrt
the metric defined by the Hilbert norm (which is uniquely determined
by standard Hilbert space theory). 
\end{rem}

\subsection{\label{sub:The-nonnormalized-setting} The non-normalized KRF as
a gradient flow}

Next we briefly consider the setting when the cohomology class $T$
is preserved by the non-normalized KRF, i.e. $\frac{\text{1}}{\beta}c_{1}(K_{X})+[\theta_{\beta}]\in H^{1,1}(X,\R)$
is trivial (as in Section \ref{sub:The trivial case}). Then  $\theta=dd^{c}f$
for a smooth function $f$ on $X.$ Introducing the functional 
\[
\mathcal{F}(\mu):=\int f\mu
\]
on $\mathcal{P}(X)$ whose differential at $\mu$ may be identified
with the function $f,$ all the results above still apply with $E_{\theta}(\mu)$
replaced by $\mathcal{F}(\mu).$ In particular, as we next explain
this leads to gradient flow representations of the Hele-Shaw flow,
as well as Hamilton-Jacobi equations, which appear to be new.

\subsubsection{The Hele-Shaw flow}

In particular, we have the following result where $H^{-1}(X)$ denotes
the Hilbert space of all signed measures on $X$ with finite (logarithmic)
energy equipped with the Dirichlet norm. 
\begin{thm}
Let $(X,\omega)$ be a Riemann surface with a normalized area form
$\omega$ and $p$ a given point on $X.$ Denote by $\Omega(t)$ the
corresponding weak Hele-Shaw flow of increasing domains in $X,$ injected
at $p.$ Then the corresponding family 
\[
\mu(t):=1_{X-\Omega^{(\lambda(t))}}(t+1)\omega_{0}
\]
of probability measures on $X$ is the gradient flow of the lsc convex
functional $\tilde{\mathcal{F}}$ on the Hilbert space $H^{-1}(X)$
obtained by extending $\mathcal{F}$ by infinity from $\mathcal{P}(X)\cap H^{-1}(X).$
In particular, $\mathcal{F}(\mu(t))$ is strictly decreasing along
the flow.\end{thm}
\begin{proof}
This is proved precisely in Theorem \ref{thm:hilb gradient flow in normalized setting}.
Even if $f$ is not smooth in this setting, the general results about
gradient flows in Hilbert spaces still apply as $f$ is lsc and the
corresponding functional is convex.
\end{proof}

\subsubsection{Hamilton-Jacobi equations}

Next we turn to the setting of Hamilton-Jacobi equations, using the
notation in Section \ref{sec:Relations-to-viscosity}. We will denote
by $\mathcal{C}_{\Lambda}^{+}$ the subspace of all smooth and strictly
convex functions $\psi$ in $\mathcal{C}_{\Lambda}$ and by $\mbox{Ent}(\mu|\nu)$
the entropy of a measure $\mu$ relative to another measure $\nu.$
We equip the space $\mathcal{C}_{\Lambda}^{+}$ with the Riemannian
metric induced from the Dirichlet type metric \ref{eq:def of Dirichlet type metric}
under the Legendre transform.
\begin{prop}
\label{prop:The-perturbed-Hamilton-Jacob as gradient}The perturbed
Hamilton-Jacobi equation \ref{eq:hj equation with nonlinear visc}
with initial data in $\mathcal{C}_{\Lambda}^{+}$ is the gradient
flow of the following functional on the Riemannian manifold $\mathcal{C}_{\Lambda}^{+}$\emph{:}
\[
\mathscr{F}_{\beta}(\psi):=\frac{1}{\beta}\mbox{Ent}(dy|MA(\psi))+\mathscr{E}_{H}(\psi),\,\,\,\mathscr{E}_{H}(\psi):=\int_{\R^{n}/\Lambda}H(\nabla\psi(y))dy,
\]
where $\nabla\psi$ denotes the $L^{\infty}-$Brenier gradient map.\end{prop}
\begin{proof}
As shown in Section \ref{sub:Relation-to-the-krf} the solution $\psi_{t}^{(\beta)}$
is the Legendre transform of the corresponding twisted Kähler-Ricci
flow $\phi_{t}^{(\beta)}$ in $\mathcal{C}_{\Lambda}^{+}.$ Moreover,
using formula \ref{eq:legendre relations} gives 
\[
\mbox{Ent}(MA(\phi)|dx)=\mbox{Ent}(dy|MA(\psi)),\,\,\,\,\int_{\R^{n}/\Lambda}MA(\phi)H=\int_{\R^{n}/\Lambda}H(\nabla\psi(y))dy
\]
and hence the result follows from the fact that the twisted KRF is
the gradient flow wrt the Dirichlet type metric of the functional
$\mbox{Ent}(MA(\phi)|dx)/\beta+\mathcal{F}(MA(\phi)).$\end{proof}
\begin{cor}
The functional $\mathscr{E}_{H}$ is decreasing along the viscosity
solution of the Hamilton-Jacobi equation with Hamiltonian $H$ and
initial data in $\mathcal{C}_{\Lambda}.$ 
\end{cor}
Specializing to the one-dimensional case we arrive at the following
\begin{thm}
Denote by $\mu(t):=\partial^{2}\psi_{t}$ the curve in the space of
probability measures on $S^{1}$ defined by the distributional second
derivative of the unique viscosity solution $\psi_{t}$ of the Hamilton-Jacobi
equation with Hamiltonian $H.$ Then $\mu(t)$ is the gradient flow
of the functional corresponding to $\mathscr{E}_{H}$ on the space
$\mathcal{P}(S^{1})$ equipped with the Wasserstein $L^{2}-$metric.
In particular, $\mathscr{E}_{H}$ is strictly decreasing at $\psi_{t_{0}}$
unless $\psi_{t_{0}}$ is a minimizer of $\mathscr{E}_{H}$ (or equivalently:
$\mu(t_{0})$ is supported in the set where $H$ attains its absolute
minimum). \end{thm}
\begin{proof}
This is shown as in the proof of Theorem \ref{thm:hilb gradient flow in normalized setting}
using Prop \ref{prop:The-perturbed-Hamilton-Jacob as gradient} and
the observation that the Wasserstein $L^{2}-$metric corresponds under
the Legendre transform to the Dirichlet metric on the Legendre transform
side. This is well-known in the case of $\R$ and the proof in the
$S^{1}-$case can, for example, be obtained using the transformation
properties of the Otto metric (the proof will appear elsewhere). 
\end{proof}

\subsection{\label{sub:Relations-to-the} Relations to the Otto metric and stochastic
gradient flows}

Given a Riemannian manifold $(X,g)$ the Otto metric \cite{ot} is
defined on the space $\mathcal{P}^{\infty}(X)$ of all volume forms
$\mu$ in $\mathcal{P}(X)$ as follows. First note that a vector field
$V$ on $X$ induces a tangent vector on $\mathcal{P}^{\infty}(X):$
\begin{equation}
\frac{d\mu_{t}}{dt}_{|t=0}:=\frac{d}{dt}_{|t=0}((F_{t}^{V})_{*}\mu),\label{eq:action of vector field}
\end{equation}
 where $F_{t}^{V}$ denotes the one-parameter group of diffeomorphisms
of $X$ defined by the flow of $V.$ Now the Otto metric may be defined
by 
\begin{equation}
\left\langle \frac{d\mu_{t}}{dt}_{|t=0},\frac{d\mu_{t}}{dt}_{|t=0}\right\rangle _{\mu}:=\inf_{V}\int_{X}g(V,V)\mu,\label{eq:def of otto metric}
\end{equation}
 where the infimum runs over all vector fields $V$ satisfying the
equation \ref{eq:action of vector field}. In physical terms, considering
a gas of particles on $X$ distributed according to the measure $\mu,$
the norm above is the minimal kinetic energy needed to produce the
rate of change $\frac{d\mu_{t}}{dt}$ of $\mu;$ see \cite[Section 2]{ot}.

As explained in \cite{ot} the Otto metric on $\mathcal{P}^{\infty}(X)$
is, at least formally, the Riemannian metric underlying the Wasserstein
$L^{2}-$metric $d_{2}$ on $\mathcal{P}(X),$ induced by $g,$ which
may be expressed as follows on $\mathcal{P}^{\infty}(X):$ 
\[
d_{2}(\mu,\nu)^{2}:=\inf_{S}\int_{X}d_{g}(x,T(x))^{2}\mu(x),\,\,\,S_{*}\mu=\nu
\]
expressed in terms of the distance function $d_{g}$ on $X\times X$
determined by the given Riemannian metric $g,$ i.e. $d_{2}(\mu,\nu)^{2}$
is the minimal cost to transport $\mu$ to $\nu$ (the general formula
on $\mathcal{P}(X)$ employs transport plans rather than transport
maps $S$). \footnote{The argument in \cite{ot} uses that the Otto metric can be identified
with the quotient metric on $DIFF(X)/SDIFF(X,dV)$ (defined wrt to
the non-invariant $L^{2}-$metric on $DIFF(X)$ induced from $g$
under which the group of volume preserving diffeomorphisms $SDIFF(X,dV)$
acts from the right by isometries) under the submersion $S\mapsto S_{*}dV_{g}.$
A different argument, motivated by numerical applications, is given
in \cite{b-br}.} 

Now, one can envisage a generalization of the Otto metric where $g$
is allowed to depend on $\mu.$ In particular, if $X$ is a Kähler
manifold with a given Kähler class $T$ then we may simply take $g_{\mu}$
to be the unique Kähler metric in $T$ furnished by the Calabi-Yau
isomorphism, i.e. the metric $g_{\mu}$ in $T$ with volume $\mu.$ 
\begin{prop}
Let $X$ be a Kähler manifold endowed with a Kähler class $T.$ Then
the corresponding Otto type metric (obtained by replacing $g$ with
$g_{\mu}$ in formula \ref{eq:def of otto metric}) coincides with
the Dirichlet type metric defined by formula \ref{eq:def of Dirichlet type metric}
above (up to the multiplicative constant $n!).$ \end{prop}
\begin{proof}
First recall that, by Hodge theory, the infimum in formula \ref{eq:def of otto metric}
is attained precisely for $V$ of the form 
\[
V=\nabla v,\,\,\,v\in C^{\infty}(X),
\]
where $\nabla$ denotes the gradient wrt $g_{\mu}$ (where $v$ is
uniquely determined mod $\R).$ Moreover, writing $\mu=\rho dV_{g}$
the element $v\in C^{\infty}(X)/\R$ may be characterized as the unique
solution to the following continuity equation 
\[
\frac{d\rho_{t}}{dt}_{|t=0}=-\nabla\cdot(\rho\nabla v).
\]
In the Kähler setting above $\rho=1$ and hence the previous equation
is equivalent to 
\[
\frac{d\mu_{t}}{dt}_{|t=0}=-\frac{dd^{c}v\wedge\omega_{u}^{n-1}}{(n-1)!}.
\]
Accordingly, writing $\mu_{t}=\omega_{u_{t}}^{n}$ for some curve
$u_{t}$ in $\mathcal{H}$ reveals that 
\[
v=-\frac{du_{t}}{dt}_{|t=0}
\]
which concludes the proof. 
\end{proof}
A remarkable property of the Otto metric (defined wrt a fixed back-ground
metric $g$ on $X)$ is that the gradient flow of the relative entropy
$H_{dV_{g}}$ is precisely the heat (diffusion) equation. More generally,
if $G$ is a functional on $\mathcal{P}^{\infty}(X)$ then the corresponding
gradient flow is given by 
\begin{equation}
\frac{\partial\rho_{t}}{\partial t}=\nabla\cdot(\rho V_{t}),\,\,\,\,\,\,V_{t}=\nabla(\delta G){}_{|(\rho_{t}dx)},\label{eq:formal gradient flow}
\end{equation}
where $\delta G_{|\mu}$ denotes, as before, the differential of $G$
at $\mu,$ identified with a function on $X.$ In particular, if $G$
is a free energy type functional of the form \ref{eq:free energy}
then the corresponding gradient flow is the following  drift diffusion
equation (non-linear Fokker-Planck equation): 
\[
\frac{\partial\rho_{t}}{\partial t}=\frac{1}{\beta}\Delta\rho_{t}+\nabla\cdot(\rho_{t}V[\rho_{t}]),
\]
 where $V[\rho_{t}]$ is the vector field given by 
\[
V[\rho_{t}]=\nabla(\delta E)_{|(\rho_{t}dx)}.
\]
Such drift diffusion equations can often be realized as large $N-$limits
of stochastic gradient flows on the $N-$particle space $X^{N}$ of
the form 

\begin{equation}
dx_{i}(t)=-\nabla_{x_{i}}E^{(N)}(x_{1},x_{2},....,x_{N})dt+\sqrt{\frac{2}{\beta}}dB_{i}(t),\label{eq:sde general intro}
\end{equation}
where $B_{i}$ denotes $N$ independent Brownian motions on the Riemann
manifold $(X,g)$ and $E^{(N)}$ is a suitable symmetric ``microscopic''
version of $E$ (in statistical mechanical terms this expresses the
non-equilibrium dynamics of $N$ diffusing particles on $(X,g),$
at inverse temperature $\beta,$ interacting by the energy $E^{(N)}).$
This is the starting point for the stochastic dynamics approach to
the construction of Kähler-Einstein metrics introduced in \cite{b-onn}.
In particular, this leads to a new dynamic construction of (twisted)
Kähler-Einstein metrics \cite{b-l}. However, one geometric draw back
of this approach is that it requires the choice of a back-ground metric
on $X$ and hence the corresponding evolution equation is not canonical
(even if its large $t-$limit is). This motivates using the generalized
Otto type metric which amounts to coupling the back-ground metric
$g_{t}$ at time $t$ to the measure $\mu_{t}.$ As will be explained
elsewhere the latter road leads to a new microscopic stochastic approach
to the Kähler-Ricci flow where the individual particles $x_{1},...,x_{N}$
perform coupled Brownian motions defined with respect to a changing
metric which depends on the location of the whole configuration of
particles.

\end{document}